\documentclass{amsart}

\usepackage{amstext}
\usepackage{amsmath}
\usepackage{amssymb}
\usepackage{amsthm}
\usepackage{enumerate}
\usepackage{caption}
\usepackage{wrapfig}
\usepackage{bm}
\usepackage{tikz}

\usetikzlibrary{arrows}

\newtheorem{theorem}{Theorem}
\newtheorem{example}{Example}
\newtheorem{lemma}{Lemma}
\newtheorem{definition}{Definition}
\newtheorem*{def1}{Definition \(\tilde{\bf1}\)}
\newtheorem{proposition}{Proposition}

\newtheorem{remark}{Remark}
\newtheorem{question}{Question}

\newtheorem*{acknowledgement}{Acknowledgement}
\newcommand{\seg}{\mathrm{seg\,}}
\newcommand{\sag}{\mathrm{sag\,}}
\newcommand{\dist}{\mathrm{dist}}
\newcommand{\R}{\mathbb{R}}
\newcommand{\rad}{\mathrm{rad\,}}
\newcommand{\diam}{\mathrm{diam\,}}
\newcommand{\vol}{\mathrm{vol\,}}
\newcommand{\m}{\mathrm{md}}
\newcommand{\crit}{\mathcal{C}}
\newcommand{\cri}{\mathrm{cr}\,}

\begin{document}

\title{Sagitta, Lenses, and Maximal volume}

\author{Curtis Pro}
\address{{Department of Mathematics, University of Toronto, ON M5S 2E4} }
\email{cpro@math.toronto.edu}

%-Title

%+Abstract
\begin{abstract}
  We  give a characterization of critical points that allows us to define a metric invariant on  all Riemannian manifolds \(M\) with a lower sectional curvature  bound  and an upper radius bound. We show there is a uniform upper volume bound for all such manifolds with an upper bound on this invariant. We generalize  results by Grove and Petersen and by Sill, Wilhelm, and the author by showing any such \(M\) that has volume sufficiently close to this upper bound is diffeomorphic to the standard sphere \(S^{n}\) or a standard lens space \(S^n/\mathbb{Z}_m\) where \(m\in\{2,3,\ldots\}\) is no larger than an a priori constant.  
\end{abstract}
\maketitle

Given a point \(p\) in a metric space \(X\), the \textit{radius} of \(X\) is defined to be the number \[\rad X: = \inf_{p\in X} \sup_{q\in X}\dist(p,q).\] Let \(k\in\R\) and \(S^n_k\) denote the simply connected space form of constant sectional curvature \(k\). If \(M^n\) is a Riemannian manifold with sectional curvature bounded below by \(k\),  it follows from usual volume comparison that  \[\vol M\leq \vol D^n_k(\rad M)\] where \(D^n_k(r)\subset S^n_k\) denotes the disk of radius \(r\). This gives a uniform upper volume bound for the class \(\mathcal M^n_{k,r}\) of all Riemannian \(n\)-manifolds \(M\) with \(\sec M\geq k\) and \(\rad M\leq r\).

In \cite{GrovePet3}, Grove and Petersen showed  if \(r\leq\frac{1}{2}\diam S^n_k\) (\(=\infty\) if \(k\leq0\) and \(\frac{1}{2}\pi/\sqrt k\) otherwise) is a real number, then  any \(M\in \mathcal M^n_{k,r}\) with volume sufficiently close to \(D^n_k(r)\),\(\) must topologically be a sphere or real projective space.

In this work we define a  metric invariant of \(M\in \mathcal M^n_{k,r}\) denoted   by \(\sag_r M\) which we called the \(r\)-\textit{sagitta} of \(M\). The definition of \(\sag_rM\) is based on a characterization of critical points given in Proposition \ref{introprop} below and shares similarities with the following notion from classical geometry. If \(C\) is a circular arc of radius \(r\), the \textit{sagitta} of \(C\) is defined to be the distance \(h\) in the plane between the midpoint of the chord through the boundary of \(C\) and the midpoint of \(C\). \(\)

In an analogous way that \(D^n_k(r) \) serves as an extremal model for manifolds in \(\mathcal M^n_{k,r}\), given \( \tilde a,\tilde b\in S^n_k\) and \(h,r\in (0,\frac{1}{2}\diam S^n_k]\) with \(|\tilde a\tilde b| = 2(r-h), \) the lens-like set \[L^n_k(h,r):=D^n_k(\tilde a_,r)\cap D^n_k(\tilde b,r)\subset S^n_k\] (See Figure \ref{figurelens}) serves as an extremal model for manifolds in \(\mathcal M^n_{k,r}\) with the additional constraint \(\sag_rM\leq h\). More precisely, our main result is
\begin{theorem}
\label{mainthrm}
 Let \(n\geq 2,\) \(k\in \R\), and \(h,r\in (0,\frac{1}{2}\diam S^n_k]\) be any real numbers with \(h\leq r\). Let \(\mathcal M^n_{k,r,h}\) denote the class of all Riemannian \(n\)-manifolds satisfying  \[\left\{\begin{array}{lc}
k\leq \sec M \\
\rad M\leq r\\
\sag_r M \leq h

\end{array}\right.\]
then,\begin{enumerate}[(1)]\item \label{p1}For every \(M\in \mathcal M^n_{k,r,h}\) the volume of \(M\) satisfies
 \[\vol M\leq \vol L^n_k(h,r).\] 
\item  \label{p2}There is an \(\varepsilon(n,k,h,r)>0\), and an integer  \(c =c(n,k,h,r)\geq 2\) so that for every \(M\in \mathcal M^n_{k,r,h}\) if \[\vol M>\vol L_n^k(h,r) - \varepsilon,\] then \(M\) is diffeomorphic to  either
\begin{enumerate}[(a)]\item \(S^n,\)  \item\(\R P^n\) if \(n\) is even, or \item A Lens space \(S^n/\mathbb{Z}_m\) where \(2\leq m\leq c\) if \(n\) is odd.
\end{enumerate}
\item \label{p3}Each manifold \(N\) in the conclusion of Part (2) admits a sequence of metrics in \(\mathcal M^n_{k,r,h}\) with volume converging to \(\vol L^n_k(h,r)\), but  no \(N\in \mathcal M^n_{k,r,h}\) satisfies the equality \(\vol N = \vol L^n_k(h,r)\) unless \(k>0\),  \( r =\frac{1}{2}\diam S^n_k\), and \(h\) divides \(\frac{1}{2}\diam S^n_k\). 
\end{enumerate}
\end{theorem}

   Given a point \(q\in M\) and subset \(A\subset M\), by \(\Uparrow_q^A\) we mean  the set of all unit vectors in \(T_qM\) tangent to segments from \(q\) to \(p\in A\). Recall that a point \(q\in M\) is said to be critical to \(p\in M\), in the sense of \cite{GroveShio}, provided  \(\Uparrow_q^p\subset S^{n-1}_1\) is a  \(\pi/2\)-net.  Given \(p\in M\), we let \(\crit(p) \) denote all points \(q\neq p\) that are critical to \(p\).

In what follows, we will  assume \(M\) is compact,  \(\sec M\geq k\in \R\), and \(M\) is not isometric to \(S^n_k\). This implies for every \(p\in M\),  \(\exp_p\) will  be surjective on the ball \(B(o_p,\diam S^n_k)\subset T_pM\), where \(o_p\) is the origin of \(T_pM\). Metrically, we   identify \(B(o_p,\diam S^n_k)\subset T_pM\)  with the ball \(B^n_k(\diam S^n_k)\subset S^n_k\) and denote this set by \(B^n_k(p)\). We note, by \cite{GroveShio},  if \(q\in \crit(p)\) satisfies \(\dist(p,q)>\frac{1}{2}\diam S^n_k\), then \(\crit(p) =\{q\}\).

\begin{proposition}\label{introprop} Let \(p\in M\) and suppose all \(q\in \crit(p)\) satisfy \(\dist(p,q)\leq \frac{1}{2}\diam S^n_k\).  Then for each \(q\in \crit(p)\),  there is a number \(\cri_p(q)\in I :=[\dist(p,q),\frac{1}{2}\diam S^n_k]\) and for every \(v_q\in \Uparrow_p^q\), there is a closed convex subset \(D_{v_q}\subset B^n_k(p) \) so that:  \begin{enumerate}[(1)]\label{p1}\item The exponential map \begin{eqnarray*}\exp_p:\bigcap_{q\in\crit(p),v_q\in \Uparrow_p^{q} }D_{v_q}\rightarrow M\end{eqnarray*} is surjective.  
\item\label{p2} For  all \(v_q,w_q\in\Uparrow_p^q\), the sets \(D_{v_q}\) and \(D_{w_q}\) are isometric.
 \item \label{p3}If  \(\tilde \gamma_{v_q}\) is the geodesic in \(B^n_k(p)\) defined by \(\tilde\gamma_{v_q}(0)=o_p\) and \(\tilde\gamma'_{v_q}(0) = v_q\), then  if \(\cri_p(q)<\frac{1}{2}\diam S^n_k, D_{v_q}\) is the disk  given by \[D_{v_q} =D^n_k(\tilde\gamma_{v_q}(\dist(p,q)-\cri_p(q)),\cri_p(q)), \]  or  if \(\cri_p(q)=\frac{1}{2}\diam S^n_k\), \(D_{v_q}\) is the closed half space given by \[D_{v_q}=\mathrm{cl}\left(\bigcup_{t\in I}D^n_k\left(\tilde\gamma_{v_q}(\dist(p,q)-t),t\right)\right)\] 
 \item \label{p4}The number \(\cri_p(q)\) depends only on \(k\) and the distance functions from \(p\) and \(q\).
\item \label{p5}\(\cri_p(q) = \dist(p,q)\) if and only if \(q\) is a point at maximal distance from \(p\).\end{enumerate}
\end{proposition}

With this Proposition we now make the following

\begin{definition}\label{lamesagdef} Suppose \(M\) satisfies \(\sec M\geq k\) and \(\rad M\leq r\). We let \[\sag_rM:= \inf\{\dist(p,q)\mid p \text{ and } q \text{ are mutually critical and }\cri_p(q)\leq r\}\]  and call  this number  the \(r\)-\emph{sagitta of} \(M\).
\end{definition}

\begin{remark}   When \(h=r\) we have \(L^n_k(r,r) =D^n_k(r)\) and we will show  that the constant \(c \) of Theorem \ref{mainthrm} is equal to \(2\).  However, the author does not know if \(\mathcal M^n_{k,r}\subset \mathcal M^n_{k,r,r}\) and therefore Theorem \ref{mainthrm} may not give a complete generalization of Theorem A in \cite{GrovePet3}. 

In Section \ref{rsagdefsec}, we give a weaker, but more technical alternative to Definition \ref{lamesagdef}. We denote this  weaker invariant by \(\widetilde\sag_rM\) and let \(\widetilde {\mathcal M}^n_{k,r,h}\) denote the class of all Riemannian \(n\)-manifolds satisfying  \[\left\{\begin{array}{lc}
k\leq \sec M \\
\rad M\leq r\\
\widetilde\sag_r M \leq h
\end{array}\right..\]
We will show that \(\mathcal M^n_{k,r,h}\subset \widetilde{\mathcal M}^n_{k,r,h}\) and that Theorem \ref{mainthrm} actually holds for the larger class \(\widetilde{\mathcal M}^n_{k,r,h}\). In addition, we show \(\mathcal M^n_{k,r}\subset  \widetilde{\mathcal M}^n_{k,r,r}\), thus the topological statement of Theorem \ref{mainthrm} for this larger class reproduces Theorem A in \cite{GrovePet3} and the smooth statement reproduces the main theorem in \cite{PSW}.

The advantage of Definition \ref{lamesagdef} is that it requires fewer prerequisites to state, and up to a possible minor technical detail, it is the desired constraint. \end{remark}

\subsection*{Examples }

  In this subsection we give examples of Alexandrov  spaces \(X\)  that are the model spaces for manifolds in \(\widetilde{\mathcal M}^n_{k,r,h}\)  with volume almost  extremal.
More precisely, we will show the first example is the only space that can arise as a limit of manifolds \(M\in \mathcal M^n_{k,r,h}\) that  have volume converging to \(\vol L^n_k(h,r)\), while both examples can occur as such limits of manifolds in \(\widetilde{\mathcal M}^n_{k,r,h}\) . 

\subsubsection*{Notation}{(See Figures \ref{figurelens} and \ref{reflections})}

 Assume \(n\geq 2, h\leq r\in (0,\frac{1}{2}\diam S^n_k]\),  let \(\tilde a_1,\tilde a_2,\tilde q_1,\tilde q_2\in S^n_k\) such that \(|\tilde a_1\tilde a_2| = 2(r-h)\), \(|\tilde p \tilde q_i| = h\), and \(|\tilde a_i\tilde q_i | = r\) and  set \(L^n_k(h,r)=D^n_k(\tilde a_1,r)\cap D^n_k(\tilde a_2,r)\). 
If \(h<r\),  let \(H_0\)  be  the totally geodesic hyperplane given by \[H_0:=\{\tilde u\mid|\tilde u\tilde a_1| = |\tilde u\tilde a_2|\}\subset S^n_k.
\] and set  \(P\subset S^n_k\) to be any hyperplane which contains the geodesic through \(\tilde a_1\) and \(\tilde a_2\).
If \(h=r\), i.e., \(L^n_k(h,r) = D^n_k(\tilde a_1,r),\)   take \(P=H_0\)  to be any hyperplane through \(\tilde a_1 = \tilde a_2\). 
 
 Define \(R_{H_0}:S^n_k\rightarrow S^n_k\) and \(
R_{P}:S^n_k\rightarrow S^n_k\) to 
be reflections over the hyperplanes \(H_0\) and \(P\), respectively.

For any \(m\in \mathbb{Z}_+\), let \[C(n-2,m) = \{\phi_{m_i}\}_{i\in I_m}\subset O(n-1)\]be all isometries of \(S^{n-2}\) of order \(m\) that, if \(m>1,\) generate a cyclic group \(\mathbb{Z}_m\) that acts freely on \(S^{n-2}\).

 Since \(\mathbb{Z}_2\) is the only group which acts freely on even dimensional spheres,    \(C(n-2,m)=\emptyset\) if \(m> 2\) and \(n\) is even. Because of this, we will always implicitly assume \(n\) is odd whenever making reference to a \(\phi_m\in C(n-2, m)\) where \(m>2\).

\begin{figure}

\begin{tikzpicture}[line join=round,>=triangle 45,x=0.7cm,y=0.7cm]
\clip(-5,-3) rectangle (5,3);
\draw [shift={(2,0)},line width=1pt]  plot[domain=2.3:3.98,variable=\t]({1*3*cos(\t r)+0*3*sin(\t r)},{0*3*cos(\t r)+1*3*sin(\t r)});
\draw [shift={(2,0)},dotted]  plot[domain=-2.3:2.3,variable=\t]({1*3*cos(\t r)+0*3*sin(\t r)},{0*3*cos(\t r)+1*3*sin(\t r)});
\draw [shift={(-2,0)},dotted]  plot[domain=0.84:5.44,variable=\t]({1*3*cos(\t r)+0*3*sin(\t r)},{0*3*cos(\t r)+1*3*sin(\t r)});
\draw [shift={(-2,0)},line width=1pt]  plot[domain=-0.84:0.84,variable=\t]({1*3*cos(\t r)+0*3*sin(\t r)},{0*3*cos(\t r)+1*3*sin(\t r)});
\draw [shift={(0,0)},line width=1pt]  plot[domain=0:3.14,variable=\t]({0*2.23*cos(\t r)+-1*0.2*sin(\t r)},{1*2.23*cos(\t r)+0*0.2*sin(\t r)});
\draw [shift={(0,0)},line width=1pt,dotted ]  plot[domain=-3.14:0,variable=\t]({0*2.23*cos(\t r)+-1*0.2*sin(\t r)},{1*2.23*cos(\t r)+0*0.2*sin(\t r)});
\draw [rotate around={90:(-2,0)},dotted] (-2,0) ellipse (2.1cm and 0.1cm);
\draw [rotate around={-90:(2,0)},dotted] (2,0) ellipse (2.1cm and 0.1cm);
\draw (2,0) node{\(\cdot\)};
\draw (2,0) node[right]{\tiny\(\tilde a_2\)};
\draw (-2,0) node{\(\cdot\)};
\draw (-2,0) node[right]{\tiny\(\tilde a_1\)};
\draw (0,0) node{\(\cdot\)};
\draw (0,.2) node{\tiny\(\tilde p\)};
\draw (1.3,0) node{\tiny\(\tilde q_1\)};
\draw (-.7,0) node{\tiny\(\tilde q_2\)};
\draw (-1,0) node{\huge\(\cdot\)};
\draw (1,0) node{\huge\(\cdot\)};
\draw (-5,0) --node[midway,above]{\tiny\(r\)}(-2,0);
\draw (0,0) --node[midway,above]{\tiny\(h\)}(1,0);
\draw (-3,-2.2) node{\tiny\(D^n_k(\tilde a_1,r)\)};
\draw (3,-2.2) node{\tiny\(D^n_k(\tilde a_2,r)\)};
\end{tikzpicture}
\caption{\(L^n_k(h,r)\) depicted as the bold region above.   \label{figurelens}  }
\end{figure}
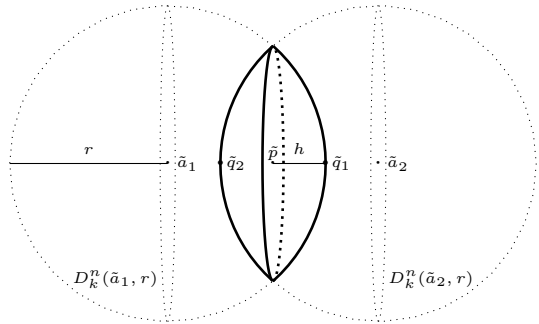

\begin{example}\label{example1} 
 
If \(n\geq2,\) any \(\phi_{m}\in C(n-2,m)\), viewed as an isometry of \(S_0\), extends to an isometry of \(L^n_k(h,r)\) that fix the points \(\tilde q_1\) and \(\tilde q_2\). Therefore, we set  \[L^n_k(h,r,\phi_{m}):= L^n_k(h,r)/(\tilde u\sim R_{H_0}\circ \phi_{m}(\tilde u)) \text{ with }\tilde u\in \partial L^n_k(h,r)\]  and note \(L^n_k(h,r,\phi_{m})\) is an Alexandrov space with curvature bounded below by \(k\), and volume equal to \(\vol L^n_k(h,r)\).   
\end{example} 
\begin{remark}  If \(n\geq2\) and \(\phi_m\in C(n-2,m)\), by scaling \(S_0\) to have constant curvature 1, we can identify the unit sphere \(S^{n}\) in \(\mathbb{R}^{n+1}\) with the spherical join \(S^1\ast S_0\) where \(S^1\) is the unit circle in \(\mathbb{C}\).  We can then  identify \(L^{n}_k(h,r,\phi_{m})\) as the fundamental domain of the free and orthogonal \(\mathbb{Z}_m\)-action on \(S^{n}\) generated by \(\psi\in O(n+1)\) where \(\psi(z,x) = (e^{\frac{2\pi i}{m}}z,\phi_m(x))\) (see Section \ref{convergencesection}). If \(m=2\), \(L^n_k(h,r,\phi_{m})\cong \R P^n\) and if \(m=1\), the degenerate Lens space \(L^n_k(h,r,\mathrm{id})\) is topologically a sphere.   We also note that \(L^n_k(r,r,\mathrm{id})\) is the ``Curvature \(k\) Purse" denoted \(P^n_{k,r}\) and \(L^n_k(r,r,-\mathrm{id})\cong\mathbb{R}P^n\) is the ``Curvature \(k\) Crosscap" denoted \(C^n_{k,r}\)  that were constructed in \cite{GrovePet3}. 

\end{remark}

\begin{example}
 (see Figure \ref{figureP}) Define \[P^n_k(h,r):= L^n_k(h,r)/(\tilde u\sim R_P(\tilde u)) \text{ where }  \tilde u\in \partial L^n_k(h,r)\] and note \(P^n_k(h,r)\cong S^n\) is an Alexandrov space with curvature bounded below by \(k\), and volume equal to \(\vol L^n_k(h,r)\). We remark that when \(h=r\), this example coincides with \(L^n_k(r,r,\mathrm{id})\).  \end{example}

  In Section \ref{section1}  we establish Proposition \ref{introprop}. In Section \ref{rsagdefsec}, we define \(\widetilde\sag_rM\). In Section \ref{volumesection} we establish the volume bound in Part \ref{p1} of Theorem \ref{mainthrm}. In Section \ref{convergencesection}, using the same ideas as presented in \cite{GrovePet3}, we prove the following convergence theorem \begin{theorem}\label{convergethrm} Let \(n\geq 2\), \(k\in\R\), and  \(h,r\in(0,\frac{1}{2}\diam S^n_k]\) with \(h\leq r\) be real numbers. There is an integer \(c = c(n,k,h,r)>0\) so that if \(\{M_i\}\subset\widetilde{\mathcal M}^n_{k,r,h}\) is a sequence of Riemannian \(n\)-manifolds with  \(\vol M_i\rightarrow \vol L^n_k(h,r)\),  for some \(m\in \{1,\ldots,c\}\)   a subsequence of \(\{M_i\}\) must converge to either \(P^n_k(h,r)\) or \(L^n_k(h,r,\phi_{m})\) for some \(\phi_{m}\in C(n -2,m)\).\end{theorem} The topological conclusion of Part \ref{p2} of Theorem \ref{mainthrm}, will then following from this theorem and Perelman's Stability Theorem \cite{Pstab,VK}. In Section 4 we establish Part \ref{p3} of Theorem \ref{mainthrm}. 

\begin{figure}
\begin{tabular}{cc}
\begin{tikzpicture}[line cap=round,line join=round,x=.7cm,y=.7cm]
\clip(-2.8,-3) rectangle (3,2.8);
\draw (-0.74,-2.55) node{\small\(H_0\)};
\draw (1.9,1.8) node{\small\(S_0\)};
\draw[stealth-] (-.18,1.3) .. controls +(2,0) and +(-2,0) .. +(1.7,.5);

\draw(1.6,-.1)node{\small\(R_{H_0}\)} ;

\draw[arrows={|-latex'}] (.5,0.248) -- (2.2,0.2);
\draw[arrows={latex'-|}] (-2,0.32) -- (-.6,0.28);
\fill[fill=black,fill opacity=0.05] (-0.34,2.8) -- (0.24,2.45) -- (0.26,-2.38) -- (-0.26,-2.92) -- cycle;
\draw [shift={(2,0)}]  plot[domain=2.3:3.98,variable=\t]({1*3*cos(\t r)+0*3*sin(\t r)},{0*3*cos(\t r)+1*3*sin(\t r)});
\draw [shift={(-2,0)}]  plot[domain=-0.84:0.84,variable=\t]({1*3*cos(\t r)+0*3*sin(\t r)},{0*3*cos(\t r)+1*3*sin(\t r)});
\draw [shift={(0,0)}]  plot[domain=0:3.14,variable=\t]({0*2.23*cos(\t r)+-1*0.2*sin(\t r)},{1*2.23*cos(\t r)+0*0.2*sin(\t r)});
\draw [shift={(0,0)},dotted]  plot[domain=-3.14:0,variable=\t]({0*2.23*cos(\t r)+-1*0.2*sin(\t r)},{1*2.23*cos(\t r)+0*0.2*sin(\t r)});
\draw [shift={(-2.02,0.07)},dotted]  plot[domain=0:0.75,variable=\t]({1*3.02*cos(\t r)+0.04*0.11*sin(\t r)},{-0.04*3.02*cos(\t r)+1*0.11*sin(\t r)});
\draw [shift={(-2.02,0.07)}] plot[domain=5.36:6.28,variable=\t]({1*3.02*cos(\t r)+0.04*0.11*sin(\t r)},{-0.04*3.02*cos(\t r)+1*0.11*sin(\t r)});
\draw [shift={(1.93,-0.13)},dotted]  plot[domain=2.2:3.13,variable=\t]({1*2.94*cos(\t r)+0.06*0.11*sin(\t r)},{-0.06*2.94*cos(\t r)+1*0.11*sin(\t r)});
\draw [shift={(1.93,-0.13)}] plot[domain=3.13:3.9,variable=\t]({1*2.94*cos(\t r)+0.06*0.11*sin(\t r)},{-0.06*2.94*cos(\t r)+1*0.11*sin(\t r)});

\draw (-0.34,2.8)-- (0.24,2.45);
\draw (0.24,2.45)-- (0.26,-2.38);
\draw (0.26,-2.38)-- (-0.26,-2.92);
\draw (-0.26,-2.92)-- (-0.34,2.8);
\end{tikzpicture}
 &
\begin{tikzpicture}[line cap=round,line join=round,>=triangle 45,x=.7cm,y=.7cm]
\clip(-5,-3) rectangle (3,3.2);

\draw (-2,-.5) node{\small\(P\)};
\draw (1.9,0.52) node{\small\(R_{P}\)};
\draw (2.2,-.15) node{\large\(\cdot\)};
\draw (2.6,-.3) node{\tiny\(\tilde a_2\)};
\draw (-1.8,.03) node{\large\(\cdot\)};
\draw  (-2,.29)  node{\tiny\(\tilde a_1\)};
\draw[arrows={|-latex'}] (1.2,0.2) -- (1.2,1.5);
\draw[arrows={|-latex'}] (1.2,-0.5) -- (1.2,-1.6);
\fill[fill=black,fill opacity=0.05] (-2.73,-0.03) -- (2.07,-0.22) -- (2.77,-0.03) -- (-1.37,0.14) -- cycle;

\draw [shift={(2,0)}]  plot[domain=2.3:3.98,variable=\t]({1*3*cos(\t r)+0*3*sin(\t r)},{0*3*cos(\t r)+1*3*sin(\t r)});
\draw [shift={(-2,0)}]  plot[domain=-0.84:0.84,variable=\t]({1*3*cos(\t r)+0*3*sin(\t r)},{0*3*cos(\t r)+1*3*sin(\t r)});
\draw [shift={(0,0)}]  plot[domain=0:3.14,variable=\t]({0*2.23*cos(\t r)+-1*0.2*sin(\t r)},{1*2.23*cos(\t r)+0*0.2*sin(\t r)});
\draw [shift={(0,0)},dotted]  plot[domain=-3.14:0,variable=\t]({0*2.23*cos(\t r)+-1*0.2*sin(\t r)},{1*2.23*cos(\t r)+0*0.2*sin(\t r)});
\draw [shift={(-2.02,0.07)},dotted]  plot[domain=0:0.75,variable=\t]({1*3.02*cos(\t r)+0.04*0.11*sin(\t r)},{-0.04*3.02*cos(\t r)+1*0.11*sin(\t r)});
\draw [shift={(-2.02,0.07)}] plot[domain=5.36:6.28,variable=\t]({1*3.02*cos(\t r)+0.04*0.11*sin(\t r)},{-0.04*3.02*cos(\t r)+1*0.11*sin(\t r)});
\draw [shift={(1.93,-0.13)},dotted]  plot[domain=2.2:3.13,variable=\t]({1*2.94*cos(\t r)+0.06*0.11*sin(\t r)},{-0.06*2.94*cos(\t r)+1*0.11*sin(\t r)});
\draw [shift={(1.93,-0.13)}] plot[domain=3.13:3.9,variable=\t]({1*2.94*cos(\t r)+0.06*0.11*sin(\t r)},{-0.06*2.94*cos(\t r)+1*0.11*sin(\t r)});
\draw (-2.73,-0.03)-- (2.07,-0.22);
\draw (2.07,-0.22)-- (2.77,-0.03);
\draw (2.77,-0.03)-- (-1.37,0.14);
\draw (-1.37,0.14)-- (-2.73,-0.03);

\end{tikzpicture}

\end{tabular}
\caption{Hyperplanes  \(H_0\) and \(P\).}\label{reflections}
\end{figure}
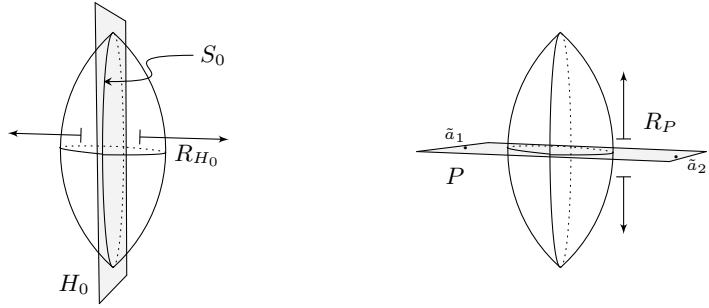

The proof of the diffeomorphism conclusion of  Part \ref{p2} of Theorem \ref{mainthrm} is possible by exploiting the geometry of these limit spaces along the same lines that were achieved in \cite{PSW}. However, we find that it is more convenient to  defer to the following

\begin{theorem}\label{diffthrm} Let \(X\) be an \(n\)-dimensional Alexandrov space with lower curvature bound \(k\) and upper diameter bound \(D\).  Let \(\{N_i\}_{i\in I}\) be a collection of smoothly and isometrically embedded \((n-2)\)-dimensional Riemannian manifolds with smooth boundary such that if \(N_i\cap N_j\neq \emptyset\), then \(N_i\cap N_j = \partial N_i = \partial N_j\). Let \(\{M_\alpha\}\) be a sequence of \(n\)-dimensional Riemannian manifolds with \(\sec M_\alpha\geq k\) and \(\diam M_\alpha\leq D\) that converge to \(X\).  If the space of directions of every point of \(X\setminus (\cup_{i\in I} N_i)\) is isometric to \(S^{n-1}_1\), then for all but finitely many \(\alpha\) and \(\beta\), \(M_\alpha\) is diffeomorphic to \(M_\beta\). \end{theorem}

By ``space of directions" at \(p\in X\), we mean the metric completion \(\Sigma_p\) of the set of geodesic directions at \(p\) with respect to the angle metric (see \cite{BGP}). Here,  we say an embedding $%
\left( S,g\right) \hookrightarrow X$ is smooth and isometric provided for
every $\varepsilon >0$ there is a neighborhood $N$ of $\Delta \left(
S\right) \subset S\times S$ so that 
\begin{equation}
\left\vert D_{V}\mathrm{dist}^{X}\left( \cdot ,\cdot \right) |_{N\setminus
\Delta \left( S\right) }-D_{V}\mathrm{dist}^{S}\left( \cdot ,\cdot \right)
|_{N\setminus \Delta \left( S\right) }\right\vert <\varepsilon
\label{smooth and isom dfn}
\end{equation}%
for all unit $V\in T\left( N\setminus \Delta \left( S\right) \right)$ . 
The proof of Theorem \ref{diffthrm}, is established in \cite{PW} which is forthcoming. 

\subsection*{Singular structure of the limits}To see that  Theorem \ref{diffthrm} applies to the conclusion of Theorem \ref{convergethrm}, we  describe the singular structure of the limit spaces. 

Let \(\pi_{m}:L^n_k(h,r)\rightarrow L^n_k(h,r,\phi_{m})\) and \(p:L^n_k(h,r)\rightarrow P^n_k(h,r)\) be the quotient maps.

When \(m=1\) and \(\phi_{m} = \mathrm{id}\), \(\pi_1(S_0)\subset L^n_k(h,r,\mathrm{id})\) is a smoothly and isometrically  embedded  round \((n-2)\)-sphere and every point in \(L^n_k(h,r,\mathrm{id})\setminus\pi_1(S_0)\) has a euclidean space of directions.

For \(m>1\),  \(\pi_{m}(S_0)\subset L^n_k(h,r,\phi_{m})\) is the quotient of \(S_0\) by the free and orthogonal action of \(\langle\phi_{m}\rangle = \mathbb{Z}_m\), so is a smoothly and isometrically embedded constant curvature Lens space \(S^{n-2}/\mathbb{Z}_m\).  Again, every point of \(L^n_k(h,r,\mathrm{id})\setminus\pi_{m}(S_0)\) has a euclidean space of directions.

\begin{figure}\centering
\vspace{-10pt}
\begin{tikzpicture}[line cap=round,line join=round,>=triangle 45,x=1.0cm,y=1.0cm]
\clip(3,-3.3) rectangle (7.5,-0.5);
\draw [shift={(4.19,0.94)}] plot[domain=-0.87:0.69,variable=\t]({-0.02*3.79*cos(\t r)+1*1.22*sin(\t r)},{-1*3.79*cos(\t r)+-0.02*1.22*sin(\t r)});
\draw [shift={(3.93,-4.37)}] plot[domain=2.27:3.83,variable=\t]({-0.02*3.73*cos(\t r)+1*1.22*sin(\t r)},{-1*3.73*cos(\t r)+-0.02*1.22*sin(\t r)});
\draw [shift={(3.98,-4.42)}] plot[domain=0.76:1.56,variable=\t]({1*3.78*cos(\t r)+0*3.78*sin(\t r)},{0*3.78*cos(\t r)+1*3.78*sin(\t r)});
\draw [shift={(4.16,0.94)}] plot[domain=4.68:5.46,variable=\t]({1*3.78*cos(\t r)+0*3.78*sin(\t r)},{0*3.78*cos(\t r)+1*3.78*sin(\t r)});
\draw [shift={(4.07,-1.74)},dotted]  plot[domain=0:1.9,variable=\t]({1*2.65*cos(\t r)+0.03*0.24*sin(\t r)},{-0.03*2.65*cos(\t r)+1*0.24*sin(\t r)});
\draw [shift={(4.07,-1.74)}] plot[domain=5.04:6.28,variable=\t]({1*2.65*cos(\t r)+0.03*0.24*sin(\t r)},{-0.03*2.65*cos(\t r)+1*0.24*sin(\t r)});
\draw (3.9,-1.08)-- (3.64,-1.48);
\draw (4.21,-1.26)-- (3.73,-2);
\draw (4.32,-1.92)-- (4.1,-2.23);

\draw[stealth-] (4.7,-1.3) .. controls +(2,0) and +(-2,0) .. +(1.7,.5);
\draw (6.9,-.8)node{\small\(p(D^P_1)\)};
\draw (6.9,-2.9)node{\small\(p(D^P_2)\)};
\draw (7.1,-1.5)node{\small\(p(S_0)\)};
\draw[stealth-] (4.7,-2.4) .. controls +(2,0) and +(-2,0) .. +(1.7,-.5);
\draw[stealth-] (5.6,-1.98) .. controls +(0,.3) and +(-1,0) .. +(1,.45);
\end{tikzpicture}
\caption{One half of \(P^n_k(h,r)\) viewed as a doubling.}\label{figureP}
\end{figure}
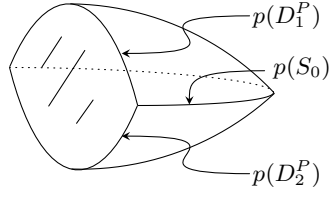

Suppose \(h< r\) and let \(P\subset S^n_k\) be the hyperplane such that \(P^n_k(h,r) = L^n_k(h,r)/(\tilde u\sim R_P(\tilde u))\), where \(\tilde u\in \partial L^n_k(h,r)\). As \(P\) contains a vector orthogonal to \(H_0\), \[p(S_0) = S_0/(\tilde u\sim R_P(\tilde u)) \subset P^n_k(h,r)\]   is an isometrically embedded \((n-2)  \)-disk of constant curvature. In addition, the \((n-2)\)-sphere \(S_P := P\cap \partial L^n_k(h,r)\)  decomposes into two \((n-2)\)-disks 
\[D^{P}_i:=P\cap \partial L^n_k(h,r)\cap \partial D^n_k(\tilde a_i,r),~~~~i=1,2,\] each with constant curvature induced from the metric on \(L^n_k(h,r)\). Because \(p|_{S_P} = \mathrm{id}\), for \(i=1,2,\) we have \(p(D^P_i)\subset P^n_k(h,r)\) is a smooth and  isometric embedding. Moreover, \[S^{n-3}\cong\partial p(S_0)=\partial p(D^P_1)=\partial p(D^P_2)\subset P^n_k(h,r).\]
Here, every point in \(P^n_k(h,r)\setminus (p(S_0)\cup p(D^P_1)\cup p(D^P_2))\) has a euclidean space of directions (See Figure \ref{figureP}).

\begin{acknowledgement}
 The author would like to sincerely thank Vitali Kapovitch and Alex Nabutovsky  for their  support and for insightful conversations about  this work. In addition, he would like to thank Fred Wilhelm for very helpful suggestions on improving the exposition of this manuscript.    \end{acknowledgement}

\section{Eccentricity and The Segment Domain}\label{section1}

If \(X\) is a metric space, points in a   disk \(D(a,r)\subset X\) are expressed as points whose distance from \(a\in X\) does not exceed \(r\). The  key ingredient in the proof of  Proposition \ref{introprop}  is that  all points in a metric disk \(D^n_k(r)\subset S^n_k\) can  also be described geometrically without reference to the distance function from the center  of this disk. More precisely, given any two distinct  points \(\tilde p,\tilde q\in S^n_k\)  all disks \(D^n_k(r)\subset S^n_k\)   with \(\tilde p\in D^n_k(r)\), \(\tilde q\in \partial D^n_k(r)\), and \(\partial D^n_k(r)\) orthogonal to the geodesic through \(\tilde p\) and \(\tilde q\) can be expressed in terms the distances from \(\tilde p \) and \(\tilde q\).    

To see this, we'll start with the motivating \(k=0 \) case. In  \(\R^2  (=S^2_0),\) Thales' theorem from classical geometry says for  any inscribed triangle \(\Delta\tilde q\tilde x\tilde q'\) in a circle \(C\) of radius \(r\),  if \(\tilde q,\tilde q'\) realize the diameter of \(C\), then \(\Delta\) will have a right angle at \(\tilde x\).  In particular if \(\tilde x\neq \tilde q\) and  \(\tilde\beta\) is the angle of \(\Delta\) at \(\tilde q\), this says for every \(\tilde x\in C\setminus{ \{\tilde q\}}\), \[\frac{\cos(\tilde\beta)}{|\tilde  q\tilde x|} = \frac{1}{2r}.\]Now suppose \(\tilde p\) is any point on the diameter between \(\tilde q\) and \(\tilde q'\). The  Law of Cosines at \(\tilde q\) says for all points \(\tilde x\in C\setminus{ \{\tilde q\}}\)\begin{eqnarray}\label{lawctha}\frac{|\tilde p \tilde x|^2 - |\tilde p\tilde q|^2}{|\tilde q\tilde x|^2} = 1 -\frac{2|\tilde p\tilde q|\cos(\tilde\beta)}{|\tilde q\tilde x|} \end{eqnarray} and so for all points \(\tilde x\in C\setminus{ \{\tilde q\}}\), we have \begin{eqnarray}\frac{|\tilde p \tilde x|^2 - |\tilde p\tilde q|^2}{|\tilde q\tilde x|^2} \equiv \frac{r-|\tilde p\tilde q|}{r}.\label{flatthale}\end{eqnarray} Since the right side of (\ref{flatthale}) is increasing in \(r\), this says given \(\tilde p\neq\tilde q\in \R^2\), if \(\lambda\) is a number less than \(1\), there is a point \(\tilde a\in \R^2\) such that all points in the circle \(C\) with radius \(|\tilde p\tilde q|/(1-\lambda)\) can be described by the set\[\left\{\tilde x\in \R^2\left|\frac{|\tilde p \tilde x|^2 - |\tilde p\tilde q|^2}{|\tilde q\tilde x|^2 } =\lambda \right.\right\}\cup\{\tilde q\}.\]When \(\lambda=1, \) this set corresponds to the line through \(\tilde q\) that is orthogonal to the line between \(\tilde p\) and \(\tilde q\).  Monotonicity of the right hand side of (\ref{flatthale}) also shows all points inside and outside of \(C\) can be described by this ratio. Moreover,  \begin{eqnarray} D\left(\tilde a,\frac{|\tilde p\tilde q|}{1-\lambda}\right)  \setminus\{\tilde q\}= \left\{\tilde x\in \R^2\,\left|\, \frac{|\tilde p \tilde x|^2 - |\tilde p\tilde q|^2}{|\tilde q\tilde x|^2}\leq \lambda\right.\right\}\label{disk}\end{eqnarray} where \(\tilde q\in \partial D\left(\tilde a,\frac{|\tilde p\tilde q|}{1-\lambda}\right)   \) and \(\tilde a\) is on the line between \(\tilde p\) and \(\tilde q\). We note if \(\lambda\geq -1\), then \(\tilde p\in D\left(\tilde a,\frac{|\tilde p\tilde q|}{1-\lambda}\right)\).  

Equation (\ref{flatthale}) is the  \(k=0 \) case of a phenomenon that holds  in \(S^n_k\) for arbitrary \(k\). Before getting to this, we first need to describe some properties of the distance modifying function \(\m_k:\R\rightarrow \R\). This function is defined as the solution to  \(y''+ky = 1\) with \(y(0)=y'(0) = 0\) and can be explicitly written as \begin{eqnarray}\label{mk}\qquad \m_k(t) = \frac{t^2}{2} +\sum_{n=2}^\infty(-k)^{n-1}\frac{t^{2n}}{(2n)!}=\left\{\begin{array}{cc}
\frac{1}{k}(1-\cos(\sqrt k\, t)) & \text{~if~} k>0 \\
\frac{t^2}{2} & \text{~if~} k=0 \\
\frac{1}{|k|}(\cosh(\sqrt {\lvert k\rvert} \,t)-1) & \text{~if~} k<0 \\
\end{array}\right. . \end{eqnarray}

The function \(\m_k\)  allows us to express a unifying formula for the Law of Cosines in \(S^n_k\): given three points \(\tilde a, \tilde x, \tilde y\in S^n_k\),  \begin{eqnarray}\m_k(c) = \m_k(a)+\m_k(b) -k\m_k(a)\m_k(b) - \m'_k(a)\m_k'(b)\cos\alpha\label{lawofcosines}
\end{eqnarray} 
where \(c = |\tilde x\tilde y|, a = |\tilde a\tilde x|, b = |\tilde a\tilde y|,\) and \(\alpha = \angle\tilde x\tilde a\tilde y \).

From this we can also use \(\m_k\) to describe the usual formulas from Trigonometry: Applying (\ref{lawofcosines}) to any three points on a single geodesic in \(S^n_k\), we have the sum formulas \begin{eqnarray}
\m_k(a+b) = \m_k(a) + \m_k(b) -k\m_k(a)\m_k(b) - \m'_k(a)\m'_k(b)\label{sumformula}
\end{eqnarray} 
and if \(c = 0\),  by (\ref{lawofcosines}) we have the Pythagorean Identities
\begin{eqnarray}
(\m_k'(t))^2 &=& \m_k(t) + \m_k(t) - k\m_k(t)\m_k(t)\label{pythagoreanidentity}\\
 &=& \m_k(t)(1+\m_k''(t))\nonumber. 
\end{eqnarray}

Next we use \(\m_k\) to express the analog of (\ref{lawctha}), and since we will refer to it often we state it as  
\begin{proposition}\label{betaprop} Let \(\tilde p,\tilde q,\tilde x\in S^n_k \)  be distinct and let \(\tilde\beta = \angle\tilde p\tilde q\tilde x\). Then \begin{eqnarray}
\frac{\m_k(| \tilde p \tilde x|) - \m_k(|\tilde p\tilde q|)}{\m_k(| \tilde q \tilde x|)} &=& \m_k''(|\tilde p\tilde q|) - \m_k'(|\tilde p\tilde q|)\frac{\m_k'(| \tilde q \tilde x|)}{\m_k(| \tilde q \tilde x|)}\cos\tilde\beta.
\end{eqnarray}   
\end{proposition} 
\begin{proof}
As \(\m_k'' + k\m_k = 1\), by (\ref{lawofcosines}) 
\begin{eqnarray*}
\m_k(| \tilde p \tilde x|) &=&   \m_k(|\tilde p\tilde q|) + \m_k(| \tilde q \tilde x|) (1-k\m_k(| \tilde p \tilde q|)) - \m_k'(|\tilde p\tilde q|)\m_k'(| \tilde q \tilde x|)\cos\tilde\beta\\
&=& \m _k(|\tilde p\tilde q|) + \m_k(| \tilde q \tilde x|)\m_k''(|\tilde p\tilde q|) - \m_k'(|\tilde p\tilde q|)\m_k'(| \tilde q \tilde x|)\cos\tilde\beta, 
\end{eqnarray*} from which the result follows.  
\end{proof}

The generalizations of (\ref{flatthale}) and (\ref{disk}) are now a result of 

\begin{proposition}\label{proposition1} Let \(r\in(0,\frac{1}{2}\diam S^n_k]\) be a real number and let \(D^n_k(r)\subset S^n_k\) be a disk of radius \(r\). Let \(\tilde p\in D^n_k(r)\) and \(\tilde q\in\partial D^n_k(r)\) be points on a geodesic through the center of \(D^n_k(r)\). Consider the constant \[\frac{\m'_k(r- \left|\tilde p\tilde q\right|)}{\m'_k(r)} = \left\{\begin{array}{cc}
\sin(\sqrt k(r-|\tilde p\tilde q|))/\sin(\sqrt k r) & \text{~if~} k>0 \\
(r-|\tilde p\tilde q|)/r & \text{~if~} k=0 \\
\sinh(\sqrt{|k|}(r-|\tilde p\tilde q|))/\sinh(\sqrt{|k|} r) & \text{~if~} k<0 \\
\end{array}\right.\]and the function \(f:S^n_k\setminus\{\tilde q\}\rightarrow \R\) defined by \[f(\tilde x) =\frac{\m_k(\left|\tilde p\tilde x\right|) -\m_k(\left|\tilde p\tilde q\right|)}{\m_k(\left|\tilde q\tilde x\right|)}. \] Then \begin{enumerate}\item for all \(\tilde x\in \partial D^n_k(r)\setminus\{\tilde q\}\) \[f(\tilde x)=\frac{\m'_k(r- \left|\tilde p\tilde q\right|)}{\m'_k(r)},\]and \item if \(B^n_k(r)\) denotes the interior of \(D^n_k(r)\)\[f|_{B^n_k(r)}<\frac{\m'_k(r- \left|\tilde p\tilde q\right|)}{\m'_k(r)}<f|_{S^n_k\setminus D^n_k(r)}\]\end{enumerate}  
 \end{proposition} 

\begin{proof} We drop the ``\(k\)" from \(\m_k\). Let \(\tilde\gamma:\R\rightarrow S^2_k\) be a geodesic and \(r\in (0,\frac{1}{2}\diam S^n_k]\)  a real number. Suppose \(\tilde a,\tilde p,\tilde q\) are points on the image of \(\tilde \gamma\) so that \(\tilde p \in D^n_k(\tilde a,r)\) and \(\tilde q \in \partial D^n_k(\tilde a,r)\). Set \(h:=|\tilde p\tilde q|\).

From (\ref{sumformula}),  
\begin{eqnarray*} \m(|\tilde p\tilde q|) =\m(h) =  \m(r-h) + \m(r) - k\m(r-h)\m(r) - \m'(r-h)\m'(r).\end{eqnarray*} 
Now take any  \(\tilde x\in\partial D^n_k(\tilde a,r)\setminus\{\tilde q\}\) and let \(\alpha:=\angle \tilde         q \tilde a \tilde x .\)   Then \(\alpha = \angle \tilde p \tilde a \tilde x\) if \(r-h> 0\) and \(\alpha = \pi-\angle \tilde     p \tilde a \tilde x\) if \(r-h<0\). Since \(\m\) is even and \(\m'\) is odd, (\ref{lawofcosines}) says \begin{eqnarray*} \m(|\tilde p\tilde x|) &=&  \m(|\tilde a \tilde p
|) + \m(r) - k\m(|\tilde a \tilde p|)\m(r) - \m'(|\tilde a \tilde p|)\m'(r)\cos\alpha\\
&=&  \m(r-h) + \m(r) - k\m(r-h)\m(r) - \m'(r-h)\m'(r)\cos\alpha.
\end{eqnarray*}
By (\ref{lawofcosines}) and (\ref{pythagoreanidentity}), \begin{eqnarray*} \m(|\tilde q\tilde x|) &=& \m(r) +\m(r) - k(\m(r))^2 - (\m'(r))^2\cos\alpha\\ &=&(\m'(r))^2(1-\cos\alpha).  \end{eqnarray*}
Combining the last three displays, it follows that,\begin{eqnarray}\frac{\m(|\tilde p\tilde x|) - \m(|\tilde p\tilde q|)}{\m(|\tilde q\tilde x|)}= \frac{\m'(r - |\tilde p\tilde q|)}{\m'(r)},\label{Q}\end{eqnarray}
proving Part (1).

Let \(\tilde\sigma:\R\rightarrow S^n_k\) be any geodesic such that \(\tilde\sigma(0) = \tilde q\) and let \(\tilde\beta \) be the angle between \(\tilde\sigma\) and \(\tilde\gamma\) at \(\tilde q\).  By Proposition \ref{betaprop},  
\begin{eqnarray}
\frac{\m(| \tilde p\tilde\sigma (t)|) - \m(|\tilde p\tilde q|)}{\m(| \tilde q \tilde\sigma (t)|)} &=& \m''(|\tilde p\tilde q|) - \m'(|\tilde p\tilde q|)\frac{\m'(| \tilde q \tilde\sigma (t)|)}{\m(| \tilde q \tilde\sigma (t)|)}\cos\tilde\beta\label{beta}.
\end{eqnarray}
  Since \(\m\) and \(\m'\) are both nonnegative on  \([0,\diam S^n_k]\), and  any point \(\tilde x\in \partial D^n_k(\tilde a,r)\setminus\{\tilde q\}\) can be connected to \(\tilde q\) with a segment that makes angle \(\leq\pi/2\) with \(\tilde \gamma\) at \(\tilde q\), (\ref{Q}) and (\ref{beta}) show \begin{eqnarray}\label{bigang}\frac{\m'(r - |\tilde p\tilde q|)}{\m'(r)}\leq \m''(|\tilde p\tilde q|).\end{eqnarray} In particular, for any \(\tilde \sigma\) with  \(\tilde \beta>\pi/2\),  (\ref{beta}) and (\ref{bigang})  say \begin{eqnarray}\frac{\m'(r-|\tilde p\tilde q|)}{\m'(r)}<\frac{\m(| \tilde p \tilde \sigma(t)|) - \m(|\tilde p\tilde q|)}{\m(| \tilde q \tilde \sigma(t)|)}\label{betabig}\end{eqnarray} for all \(t\in(0,\diam S^n_k\)]. 

On \([0,\diam S^n_k]\), \(\m\) is nondecreasing,  therefore \(\m'/\m\) is nonincreasing. This says, if \(\tilde\beta\leq\pi/2\), the right side of (\ref{beta}) will strictly increase with \(|\tilde q\tilde \sigma(t)|\). Therefore, if we choose \(t_b\in(0,\diam S^n_k\)) so that  \(\tilde\sigma(t_b)\in \partial D^n_k(r)\), we have for all \(t_{\text{int}}\in (0,t_b)\) and \(t_{\text{ext}}\in(t_b,\diam S^n_k]\)
\begin{eqnarray}\qquad\frac{\m(| \tilde p \tilde \sigma(t_{\text{int}})|) - \m(|\tilde p\tilde q|)}{\m(| \tilde q \tilde \sigma(t_{\text{int}})|)}< \frac{\m'(r-|\tilde p\tilde q|)}{\m'(r)}<\frac{\m(| \tilde p \tilde \sigma(t_{\text{ext}})|) - \m(|\tilde p\tilde q|)}{\m(| \tilde q \tilde \sigma(t_{\text{ext}})|)}.\label{betasmall}
\end{eqnarray} Equations (\ref{betabig}) and (\ref{betasmall}) complete the proof of Part (2).\end{proof}
Now notice, for any \(k\in \R\) and \(h\in(0,\diam S^n_k)\), the function \(r\mapsto \m'_k(r- h)/\m'_k(r)\) is  strictly increasing on \((0,\diam S^n_k)\). This says for any number \(\lambda\) less than  \begin{eqnarray*}\Lambda(k,h):= \lim_{r\rightarrow \diam S^n_k}\frac{ \m'_k(r- h)}{\m'_k(r)} =\left\{\begin{array}{cc}
\infty & \text{~if~} k>0 \\

e^{-\sqrt {\lvert k\rvert} h} & \text{~if~} k\leq0 \\
\end{array}\right.,\end{eqnarray*} the equation \begin{eqnarray}\label{rsol}\frac{m'_k(r- h)}{m'_k(r)} = \lambda\end{eqnarray} can be solved uniquely for \(r\).  Therefore, given \(\tilde p,\tilde q\in S^n_k\) with \(|\tilde p\tilde q|= h\in (0,\diam S^n_k)\), if \(\lambda\in[-1,\Lambda(k,h))\) and \(r_\lambda\) denotes the unique solution to (\ref{rsol}) we have 
\[D^n_k\left(\tilde a,r_\lambda\right)  \setminus\{\tilde q\}= \left\{\tilde x\in S^n_k\,\left|\, \frac{\m_k(\left|\tilde p\tilde x\right|) -\m_k(\left|\tilde p\tilde q\right|)}{\m_k(\left|\tilde q\tilde x\right|)}\leq \lambda\right.\right\}\]
where \(\tilde q\in \partial D^n_k\left(\tilde a,r_\lambda\right)   \), \(\tilde p\in D^n_k(\tilde a,r_\lambda)\) and \(\tilde a\) is on the geodesic determined by \(\tilde p\) and \(\tilde q\). 

Now, if  \((X,\dist)\) is an Alexandrov space with \(\mathrm{curv\,}X\geq k\), given \(p,q\in X\), we can use Proposition \ref{proposition1} to construct subsets of \(X\), i.e., \[ \left\{x\in X\,\left|\, \frac{\m_k(\dist(p,x)) -\m_k(\dist(p,q))}{\m_k(\dist(q,x))}\leq \lambda\right.\right\}\] that would otherwise correspond  to metric disks if \(X=S^n_k\). As in the definition of radius of \(X\), we are interested in  the smallest such subset that covers \(X\). This motivates the following 

\begin{definition}\label{d2} If \(k\in \R\) is fixed and \(X\) is an \(n\)-dimensional Alexandrov space with \(\mathrm{curv}\,X\geq k\), for \(p,q\in X\), we'll say the \(k\)-\emph{eccentricity} at \(p\) relative to \(q \) is given by \(\)\[\lambda_p(q):=\sup_{x\in M\setminus\{q\}}\frac{\m_k(\dist(p,x)) - \m_k(\dist(p,q))}{\m_k(\dist(q, x))}.\] \end{definition}
\begin{definition}\label{d3}If \(\lambda_p(q)<\Lambda(k,\dist(p,q))\),   denote by \(r_{\lambda_p(q)}\) the unique solution \(r\) to the equation \[\frac{\m'_k(r-\dist(p,q))}{\m'_k(r)} = \lambda_p(q)\, ,\] otherwise we set \(r_{\lambda_p(q)} = \infty. \) \end{definition}

Next we observe the following relationship between eccentricity and critical points
\begin{proposition}\label{criticalpointprop} Let \(M\) be a Riemannian manifold with \(\sec M\geq k\). If \(p,q\in M\), then \(\lambda_p(q)<\infty\) if and only if \(q\) is critical for the distance from \(p\).\end{proposition}

\begin{proof}

Assume \(p\neq q\). Take any segment \(\gamma_q:[0,T]\rightarrow M\) with \(\gamma_q(0)=q.\) For any \(t\in (0,T),\) by the mean value theorem, for some \(d^\ast\in[\dist(p,q),\dist(p,\gamma_q(t))] \) and \(t^\ast\in(0,t)\) we have \begin{eqnarray*}\frac{\m(\dist(p,\gamma_q(t)))-\m(\dist(p,q))}{ \m(\dist(q,\gamma_q(t)))}&=&\frac{\m'(d^\ast)}{ \m'(t^\ast)}\,\frac{\dist(p,\gamma_q(t))-\dist(p,q)}{t}.\end{eqnarray*} Using this and the first variation formula, if \(\lambda_p(q)<\infty\), then  \begin{eqnarray*}-\cos\alpha_\text{min} &=& \lim_{t\rightarrow 0^{+}}\,\frac{\dist(p,\gamma_q(t))-\dist(p,q)}{t}\\
&\leq&\lim_{t\rightarrow0^+}\,\frac{\m'(t^\ast)}{ \m'(d^\ast)}\,\lambda_p(q)\\ &=&\frac{\m'(0)}{\m'(\dist(p,q))}\,\lambda_p(q)\\ &=& 0.\end{eqnarray*} So \(q\) is a critical point for the distance from \(p\). 

Conversely, suppose \(q\) is  critical  to \(p\). Take any \(x\in M\) different from \(q\) and any segment \(\sigma_{qx}\) from \(q\) to \(x\). It follows there is a segment \(\gamma_{qp}\) from \(q\) to \(p\) so that \[\beta := \sphericalangle(\gamma_{qp}'(0),\sigma_{qx}'(0))\leq\pi/2.\] Now let \(\tilde\gamma_{\tilde q\tilde p}\) and \(\tilde \sigma_{\tilde q\tilde x}\) be   segments in \(S^n_k\) where \(|\tilde q\tilde p| = \dist(p,q)\), \(|\tilde q\tilde x | =\dist(q,x)\), and \[\tilde\beta := \sphericalangle(\tilde \gamma_{\tilde q\tilde p}'(0),\tilde \sigma_{\tilde q\tilde x}'(0)) = \beta\leq\pi/2.\]
By the hinge version of Toponogov's  Theorem, \[\dist(p,x)\leq \lvert\tilde p\tilde x\rvert.\]
From Proposition \ref{betaprop}, it follows that\begin{eqnarray*}\frac{m(\dist(p,x)) - m(\dist(p,q))}{m(\dist(q,x))}&\leq& \frac{m(| \tilde p \tilde x|) - m(|\tilde p\tilde q|)}{m(| \tilde q \tilde x|)}\\ &=& m''(|\tilde p\tilde q|) - m'(|\tilde p\tilde q|)\frac{m'(| \tilde q \tilde x|)}{m(| \tilde q \tilde x|)}\cos\tilde\beta\\ 
&\leq&m''(|\tilde p\tilde q|). \end{eqnarray*}So \(\lambda_p(q)<\infty.\)    
\end{proof}

Following \cite{GrovePet3}, for \(p\in M\), the star convex region of \(B^n_k(p) \subset T_pM\) that is bounded by the tangent cut locus of \(p\) and contains the origin will be called the \emph{segment domain} at \(p\).  We denote this set as \(\seg(p)\) and note\begin{eqnarray*}\text{seg}(p)=\{v\in B^n_k(p)\mid\text{exp}_p(tv):[0,1]\rightarrow M\text{ is a segment} \}.\end{eqnarray*}The map  \(\exp_p:\seg(p)\rightarrow M\) is surjective and if \(\seg(p)\) has the metric induced from \(B^n_k(p)\), by Topnogov's theorem, it is also 1-Lipschitz.

Recall that we use \(\Uparrow_p^q\) to denote the set of unit tangent vectors in \(T_pM\) which are tangent to segments from \(p\) to \(q\). 
\begin{definition}\label{clerical}
Let \(\tilde p\) denote the origin of \(\seg(p)\).  For each \(q\in \crit(p)\), and for each \(v\in\Uparrow_p^q\), let  \(\tilde\gamma_{v}\) be the geodesic in \(B^n_k(p)\) such that
\(\tilde\gamma_v (0) = \tilde p \) and \(\tilde \gamma_{v}'(0)=v\). If \(r_{\lambda_p(q)}<\frac{1}{2}\diam S^n_k\), set \(\tilde a_{v}:= \tilde\gamma_{v}(\dist(p,q)-r_{\lambda_p(q)})\) and let
\[D^n_{k}(\tilde a_{v},r_{\lambda_p(q)})\subset B^n_k(p)\] denote the disk centered at \(\tilde a_{v}\) with radius \(r_{\lambda_p(q)}\).
 If \(r_{\lambda_p(q)}=\frac{1}{2}\diam S^n_k \), set  \(\tilde q_v:= \tilde \gamma_v(\dist(p,q))\) and let \begin{eqnarray*}H_{v} &:=&\{\tilde x\in B^n_k(p)\mid \angle \tilde p\tilde q\tilde x\leq \pi/2\}\\ &=&\mathrm{cl}\left(\bigcup_{t\in I}D^n_k\left(\tilde\gamma_{v_q}(\dist(p,q)-t),t\right)\right) \end{eqnarray*} where \(I=[\dist(p,q),\frac{1}{2}\diam S^n_k]\).  \end{definition}

\begin{lemma} \label{segcontain} Given \(q\in \crit(p)\), if \(r_{\lambda_p(q)}<\frac{1}{2}\diam  S^n_k\) and      \(\dist(p,q)<r_{\lambda_p(q)}\), then \[\seg(p)\subset\bigcap_{v\in \Uparrow_p^q}D^n_{k}(\tilde a_v,\rad_p(q)).\]
\end{lemma}

\begin{proof}

Let \(\tilde p\) be the origin of \(\seg(p)\),  \(v\in \Uparrow_p^q\), and set \(\tilde q_v:=\dist(p,q)v\).   
To show that \(\seg(p)\subset D^n_{k}(\tilde a_v,r_{\lambda_p(q)})\), by Proposition \ref{proposition1}, it will suffice to show for any \(\tilde x\in \seg(p)\), \begin{eqnarray}\label{lambda} \frac{\m(|\tilde p \tilde x|) - \m(|\tilde p\tilde q_v|)}{\m(|\tilde q\tilde x|)}\leq \frac{\m'(r_{\lambda_p(q)}-\dist(p,q))}{\m'(r_{\lambda_p(q)})} = \lambda_p(q) .\end{eqnarray} Since \(r\mapsto \m'(r-|\tilde p\tilde q_v|)/\m'(r)\) is monotone increasing and vanishes when \(r=\dist(p,q),\) it follows from the assumption \(\dist(p,q)<r_{\lambda_p(q)}\) that \(\lambda_p(q)>0\). In particular, (\ref{lambda}) holds for all \(\tilde x\) such that  \(|\tilde p \tilde x|\leq |\tilde p\tilde q_v|\).    

Assume that \(|\tilde p\tilde x|>|\tilde p\tilde q_v|\). Let \(x\in M\) be such that  \(\exp_p(\tilde x) = x\). 
By the hinge version of Toponogov's  Theorem, \[\dist(q,x)\leq |\tilde q_v\tilde x|.\]
Therefore, \begin{eqnarray*} 0&<& \frac{\m(|\tilde p \tilde x|) - \m(|\tilde p\tilde q_v|)}{\m(| \tilde q_v\tilde x|)}\\
&\leq& \frac{\m(\dist(p,x)) - \m(\dist(p,q))}{\m(\dist(q,x))}\\
&\leq &\lambda_p(q) \end{eqnarray*} as desired. 
\end{proof} 

\begin{lemma}\label{halfspacelemma} Given \(q\in \crit(p)\), if \(\dist(p,q)<r_{\lambda_p(q)}=\frac{1}{2}\diam  S^n_k\)  then \[\seg(p)\subset\bigcap_{v\in \Uparrow_p^q}H_{v}\]\end{lemma}

\begin{proof} Take any \(\tilde p,\tilde q\in S^n_k\) with \(|\tilde p\tilde q|< \frac{1}{2}\diam S^n_k\). Suppose that  \(\tilde \Delta\tilde q\tilde p\tilde x\) and \(\tilde\Delta \tilde q\tilde p x^*\) are triangles in \(S^n_k\) with \(|\tilde p\tilde x| = |\tilde p x^*|\) and \(|\tilde q\tilde x|\leq|\tilde qx^*|\). If \(\tilde \angle\tilde p\tilde q\tilde x\leq \pi/2\), then \(\tilde \angle\tilde q\tilde px^*\leq\pi/2\). To see this, since \(|\tilde p\tilde q|< \frac{1}{2}\diam S^n_k\),  either \[\frac{\m(|\tilde px^*|)-\m(|\tilde p\tilde q|)}{\m(|\tilde q x^*|)}\leq 0<\m''(|\tilde p\tilde q|),\]or by Proposition \ref{betaprop},   \begin{eqnarray*}\frac{\m(|\tilde px^*|)-\m(|\tilde p\tilde q|)}{\m(|\tilde q x^*|)}&\leq& \frac{\m(|\tilde p\tilde x|)-\m(|\tilde p\tilde q|)}{\m(|\tilde q \tilde x|)} \\
&= & \m''(|\tilde p\tilde q|)-\m'(|\tilde p\tilde q|)\frac{\m'(|\tilde q\tilde x|)}{\m(|\tilde q\tilde x|)}\cos\tilde  \angle\tilde p\tilde q\tilde x\\ &\leq& \m''(|\tilde p\tilde q|).\end{eqnarray*} In either case, these inequalities combined with Proposition \ref{betaprop} imply \(\tilde \angle\tilde p\tilde q x^*\leq \pi/2.\)  

Now suppose that \(\tilde p\) is the origin of \(\seg(p)\). Take any \(x^*\in \seg(p)\) and let \(x\in M\) satisfy \(\exp_p(x^*) = x\). For each \(v\in\Uparrow_p^q \), let \(\tilde q_v:=\dist(p,q)v\).  Since \(q\) is critical to \(p\), there is a triangle \(\Delta pqx\) in \(M\) such that \(\angle pqx\leq \pi/2\).  By Toponogov's theorem,  if \(\tilde x\) is a point so that \(\tilde\Delta\tilde q_v\tilde p\tilde x\) is a comparison triangle for \(\Delta qpx\), then \(\tilde\angle\tilde p\tilde q_v\tilde x\leq \angle pqx\leq \pi/2. \) In addition, by the hinge version of Toponogov's theorem, \(\dist(q,x) = |\tilde q_v\tilde x|\leq |\tilde q_v x^*|\). Therefore, by the above, it follows that for all \(v\in \Uparrow_p^q\), \[\tilde \angle \tilde p\tilde q_v x^*\leq\pi/2\]completing the proof.        \end{proof}

\begin{lemma}\label{stupid} Given \(p,q\in M\), \(r_{\lambda_p(q)}\leq\dist(p,q)\) if and only if \(q\) is a point at maximal distance from \(p\). 
\end{lemma}
\begin{proof} Note \(q\) is at maximal distance from \(p\) if and only if for all \(x\in M\) \[\frac{\m(\dist(p,x))-\m(\dist(p,q))}{\m(\dist(q,x))}\leq0,\] which is equivalent to \(\lambda_p(q)\leq 0\) which is  the same as \(r_p(q)\leq\dist(p,q)\). 
\end{proof} 
  
Motivated by this Lemma, we make

\begin{definition}\label{d5} Given \(p,q\in M\) we define the \emph{critical radius} at \(q\) for the distance from \(p\) to be the number \[\cri_p(q):=\max\{\dist(p,q),r_{\lambda_p(q)}\}.\] 
\end{definition}  

\begin{proof}[\bf{Proof of Proposition \ref{introprop}}] The proof of Parts (\ref{p1}) through (\ref{p3}) follow from Lemmas \ref{segcontain} and \ref{halfspacelemma} and Parts (\ref{p4}) and (\ref{p5}) follow from Lemma \ref{stupid} and Definitions \ref{d2},\ref{d3}, and \ref{d5}. 
\end{proof}

\section{ A weaker definition of Sagitta}\label{rsagdefsec}

Then definition of \(r\)-sagitta in the introduction has the feature of being simple and at the same time uses Proposition \ref{introprop} to constrain the geometry of \(M\) in the desired way. With this definition, however, it is unclear if the condition \(\rad M\leq r\) implies \(\sag_r M\leq r\).  This would be the case if there was a positive answer to the following

\begin{question} Given a compact Riemannian manifold \(M\), if \(p\in M\) realizes the radius, does there exist a point \(q\in M\) at maximal distance from \(p\) such that \(p\) is critical for the distance from \(q\)?
\end{question}

 If we denote by \(A(p)\) to be the set of all points at maximal distance from \(p\),  we do have a partial answer to the question above.

\begin{proposition}\label{radcriticallemma} Let \(p\in M\) be a point that realizes the radius of \(M\). Then \(p\) is a critical point for the distance from \(A(p)\).  
\end{proposition}Now consider the set  \begin{eqnarray}\label{Aset} A_{h,r}(p) = \{q\in\crit(p)\mid\dist(p,q)\leq h  \text{ and } \cri_p(q)\leq r\}.\end{eqnarray}By Proposition \ref{introprop} if \(p\) realizes the radius of \(M\), \(A_{h,r}(p)\) coincides with \(A(p)\) when \(h=r=\rad M\). Therefore, by Proposition \ref{radcriticallemma}, we can make the following

\begin{def1}Suppose \(M\) satisfies \(\sec M\geq k\) and \(\rad M\leq r\). We let \[\widetilde\sag_rM:= \inf\{h\mid p\in\crit(A_{r,h}(p)) \text{ for some } p\in M\}\]  and call  this number  the modified \(r\)-\emph{sagitta of} \(M\).

\end{def1}
Clearly \(\mathcal M^n_{k,r,h}\subset \widetilde{\mathcal M}^n_{k,r,h}\) and by Proposition \ref{radcriticallemma}, \(\mathcal M^n_{k,r}\subset\widetilde{\mathcal M}^n_{k,r,r}\) as desired. As we will now only consider \(\widetilde\sag_rM\) for the proof of Theorem \ref{mainthrm}, in all that follows we drop the ``\(\sim"\) from our notation. 

\begin{proof}[\bf{Proof of Proposition} \ref{radcriticallemma}]   
For any \(x\in M\), let \(\sphericalangle\) denote the induced metric on unit tangent sphere \(U_xM = S^{n-1}_1\). Set \(R = \rad M\) and suppose for a contradiction that there is a vector \(g\in U_pM\) such that \[\sphericalangle(g,\Uparrow_p^{A(p)})>\pi/2.\] 
As \(\Uparrow_p^{A(p)}\subset U_pM\) is compact, there is a \(\theta_0>0\) so that  \[\sphericalangle(g,\Uparrow_p^{A(p)})\geq\pi/2+\theta_0.\] 
By definition of \(A(p)\) and continuity of \(\exp_p\), there is an \(r_0>0\) so that  any normal geodesic \(\sigma_p\) emanating from \(p\) that satisfies \[\sphericalangle(g,\dot\sigma_p(0))\leq \pi/2+\theta_0/2 \]
cannot be distance minimizing on \([0,R-r_0]\).

Let \(\tilde p\in S^2_k\) and let \(\tilde \gamma_{\tilde p}\) be a geodesic satisfying \(\tilde\gamma_{\tilde p}(0) = \tilde p\). By first variation there is a \(t_0>0\) so that if \(\tilde x\in S^2_k\) makes an angle less than \(\pi/2 -\theta_0/2 \) with \(\tilde\gamma_{\tilde p}\) at \(\tilde p\), then
\[\lvert\tilde\gamma_{\tilde p}(t)\tilde x\rvert<\lvert\tilde p\tilde x\rvert\]
for all \(t\in (0,t_0)\).

Now let \(\gamma_p:\R\rightarrow M\) be a geodesic that satisfies \(\gamma_p(0) = p\) and \(\dot\gamma_p(0) = -g\). 
For any point \(x\in M\) for which there is a segment \(\gamma_{px}\) from \(p \) to \(x\) so that \(\sphericalangle (-g,\dot\gamma_{px}(0))<\pi/2-\theta_0/2\), by the hinge version of Toponogov's Theorem \[\dist(\gamma_{p}(t),x)<\dist(p,x)\leq R\] for all \(t\in(0,t_0)\). 

On the other hand we know any point \(x\in M\) for which there is a segment \(\gamma_{px}\) from \(p \) to \(x\) with \(\sphericalangle(-g,\dot\gamma_{px}(0))\geq\pi/2-\theta_0/2,\) satisfies \[\dist(p,x)<R-r_0.\] 

By the triangle inequality, for any point \(x\in M\) the point \(\gamma_{p}(c)\) where \(c\in(0,\min\{t_0,r_0\})\)  satisfies \(\dist(\gamma_{p}(c),x)<R\) and so \(M\) must have radius less than \(R\).    
\end{proof} 
\section{Volume Bound}\label{volumesection}
Next we aim to prove  the volume inequality in Part 1 of Theorem \ref{mainthrm}. 
We begin with a volume inequality in \(S^n_k\) that follows from the result of the Appendix  in \cite{GrovePet2}, and is the complimentary version of Inequality  \((1.4)\) in \cite{GrovePet3}.
 
 \begin{lemma}\label{gpvolsnk}For a point \(\tilde p\in S^n_k\), let \(S(\tilde p,R) := \{\tilde a\in S^n_k\mid |\tilde p\tilde a|=R\}\) denote the metric sphere of radius \(R\) at \(\tilde p\).  Given a real number \(R\leq \frac{1}{2}\diam S^n_k\), if \(\tilde C\subset S(\tilde p,R)\) and \(\Uparrow_{\tilde p}^{\tilde C}\subset U_{\tilde p}S^n_k\)  forms a \(\pi/2\)-net, then for any \(r>0\), \begin{eqnarray}\label{volinequality}\vol \left(\bigcap_{\tilde c\in \tilde C} D^n_k(\tilde c,r)\right)\leq\vol(D^n_k(\tilde c_1,r)\cap D^n_k(\tilde c_2,r))\end{eqnarray}where \(\tilde c_1,\tilde c_2\in S(\tilde p,R)\) and \(\angle c_1 p c_2 = \pi\). Equality in (\ref{volinequality}) occurs only when \(\tilde C = \{\tilde c_1,\tilde c_2\}\) with \(\angle \tilde c_1 \tilde p\tilde c_2=\pi\). \end{lemma}
\begin{proof} We can assume \(r\geq R\). Note first \begin{eqnarray}\label{third}\bigcap_{\tilde c\in \tilde C}D^n_k(\tilde c,r)\subset D^n_k(\tilde p, r).\end{eqnarray} If not, then there is a point \(\tilde x\in\cap_{\tilde c\in \tilde C}D^n_k(\tilde c,r)\),  with \(|\tilde p\tilde x|>r\). Then we can find a segment  \(\tilde\sigma\)  from \(\tilde p\) to \(\tilde x\) such that the angle at \(\tilde p\) between \(\tilde \sigma\) and the segment from \(\tilde p\) to any \(\tilde  c\in \tilde C\) must be strictly less than \(\pi/2\). This contradicts the assumption \(\Uparrow_{\tilde p}^{\tilde C}\subset U_{\tilde p}S^n_k\)  forms a \(\pi/2\)-net.

So,\begin{eqnarray*}\vol \left(\bigcap_{\tilde c\in \tilde C} D^n_k(\tilde c,r)\right ) = \vol D^n_k(\tilde p,r)-\vol\left(\bigcup_{\tilde c\in \tilde C} D^n_k(\tilde p,r)\setminus D^n_k(\tilde c,\tilde p)\right). \end{eqnarray*}
For each \(v_{\tilde c} \in \Uparrow_{\tilde p}^{\tilde C}\), let \(-v_{\tilde c}\) denote the direction opposite \(v_{\tilde c}\) and let \(B_s(-v_{\tilde c}, \theta)\) denote the metric ball centered at \(-v_{\tilde c}\) of radius \(\theta\) in  \(U_{\tilde p}S^{n}_k\). Then note for each \(h\in [r,r+R], \) there is a \(\theta_h\geq0\) and a continuous function \(\rho_h:[0,\theta_h]\rightarrow \R\) so that for each \(\tilde c\in \tilde C\), \begin{eqnarray*} D^n_k(\tilde p,r)\setminus D^n_k(\tilde c,r) &=& \bigcup_{h\in[r,r+R]}S(\tilde c,h)\cap D^n_k(\tilde p,r)\\ &=&\bigcup_{\theta\in[0,\theta_h]}\rho_h(\theta)B_s(-v_{\tilde c},\theta).\end{eqnarray*} The result now  follows from the Appendix in \cite{GrovePet2}.   
\end{proof}

\begin{proof}[\bf{Proof of Part 1 of Theorem \ref{mainthrm}}] Let \(h,r\in (0,\frac{1}{2}\diam S^n_k]\) with \(h\leq r\), and assume \(M\in {\mathcal M}^n_{k,r,h}\). Take a point \(p\in M\) such that the distance from \(A_{r,h}(p)\) is critical at \(p\) where \(A_{r,h}(p)\) is the set defined in (\ref{Aset}). Let \(\tilde p\) be the origin of \(\seg(p)\). We now show there is a set \(\tilde C\subset S(\tilde p,r-h)\subset S^n_k\) satisfying the hypothesis of Lemma \ref{gpvolsnk}. 

  For each \(q\in A_{r,h}(p)\)   and \({v_q\in\Uparrow_p^q}\), let  \(\tilde\gamma_{v_q}\) and \( \tilde a_{v_q} \) be as in  Definition \ref{clerical}.
By assumption, \(\Uparrow_{p}^{A_{r,h}(p)}\subset U_pM\) forms a \(\pi/2\)-net and by Lemma \ref{segcontain},  \begin{eqnarray}\label{segcontainproof}\seg(p)&\subset& \bigcap_{q\in A_{r,h}(p)}D^n_k(\tilde a_{v_q},\cri_p(q)).\end{eqnarray}
By definition, for each \(q\in A_{r,h}(p)\), \begin{eqnarray*} \cri_p(q)&\leq& r,\\
\dist(p,q)&\leq& h,\end{eqnarray*} and for every \(v_q\in \Uparrow_p^q\),  \begin{eqnarray*}\tilde a_{v_q} &=&  \tilde \gamma_{v_q}(\dist(p,q) - \cri_p(q)),\\ 
\tilde p &=& \tilde\gamma_{v_q}(0).\end{eqnarray*}

For each \(q\in A_{r,h}(p)\) and \(v_q\in \Uparrow_p^q \), set \[\tilde b_{v_q}:=\tilde\gamma_{v_q}(\dist(p,q)-r),
\] and \[\tilde c_{v_q}:=\tilde\gamma_{v_q}(h-r).\] Then \(|\tilde b_{v_q}\tilde  a_{v_q}| =  r -\cri_p(q)\), and by the Triangle Inequality,\begin{eqnarray}\label{first}D^n_k(\tilde a_{v_{q}},\cri_p(q))\subset D^n_k(\tilde b_{v_q},r).\end{eqnarray}
Also, \(|\tilde b_{v_q}\tilde  p| = r-\dist(p,q)\) so,  \begin{eqnarray}\label{second}|\tilde c_{v_q}\tilde p| 
 =  r-h\leq |\tilde b_{v_q}\tilde p|.
 \end{eqnarray}

In addition since \(p\) is critical for the distance from \(A_{r,h}(p)\), by  (\ref{third}), \begin{eqnarray*}\bigcap_{q\in A_{r,h}(p)}D^n_k(\tilde b_{v_q},r)\subset D(\tilde p, r).\end{eqnarray*} 
It then follows  follows from (\ref{segcontainproof}), (\ref{first}), (\ref{second}),  and the Triangle Inequality that  \begin{eqnarray}\label{bigdisksegcontain}\seg(p)&\subset& \bigcap_{q\in A_{r,h}(p)}D^n_k( \tilde a_{v_q},\cri_p(q))\nonumber\\&\subset& \bigcap_{q\in A_{r,h}(p)}D^n_k( \tilde b_{v_q},r)\nonumber\\&\subset&\bigcap_{q\in A_{r,h}(p)}D^n_k( \tilde c_{v_q},r).\end{eqnarray}
From Lemma \ref{gpvolsnk} and Equation (\ref{second}), we have for any \(\tilde c_1,\tilde c_2\in S^n_k\) with \(|\tilde c_1\tilde c_2|=2(r-h),\) \begin{eqnarray*}\vol M &\leq& \vol\left(\seg(p)\right)\\ &\leq& \vol\left(\bigcap_{q\in A_{r,h}(p)}D^n_k( \tilde c_{v_q},r)\right)\\ &\leq& \vol(D^n_k(\tilde c_1,r)\cap D^n_k(\tilde c_2,r))\\
&=& \vol L^n_k(h,r)\end{eqnarray*} as desired. 
 \end{proof}
For future reference, we extract from Equations (\ref{second}) and (\ref{bigdisksegcontain}) in the proof above the following
\begin{proposition}\label{volcorprop}
If \(h,r\in (0,\frac{1}{2}\diam S^n_k]\)  are real numbers with \(h\leq r\) and \(M\in \mathcal M^n_{k,r,h}\), then there is a point \(p\in M\) such that if \(\tilde p\) is the origin of \(\seg(p)\), there is a \(\pi/2\)-net   \(\{\tilde c_{v_q}\}_{v_q\in \Uparrow_{p}^{A_{r,h(p)}}}\) in the metric sphere \(S(\tilde p,r-h)\) so that\begin{eqnarray*}\seg(p)&\subset&  \bigcap_{q\in A_{r,h}(p)}D^n_k( \tilde a_{v_q},\cri_p(q))\\ &\subset& \bigcap_{q\in A_{r,h}(p)}D^n_k( \tilde c_{v_q},r).\end{eqnarray*} \end{proposition}

\section{Convergence and Topological Identification}\label{convergencesection}

The goal of this section is to prove the topological version of Part 2 of Theorem \ref{mainthrm} by proving Theorem \ref{convergethrm}. Most of the ideas in this section are  taken directly from the analogous section in \cite{GrovePet3}. 

We fix \(n\geq 2\) and real numbers \(k\in \R\), \(h,r\in (0,\frac{1}{2}\diam S^n_k\)], with \(h\leq r\). Fix a  sequence \(\{M_i\}_{i=1}^\infty\subset \mathcal M^n_{k,r,h}\) of compact, Riemannian \(n\)-manifolds 
satisfying \[
 \vol M_i\rightarrow \vol L^n_k(h,r).\]

 In each \(M_i\), take a point \(p_i\in M_i\) for which \(  A_{r, h}(p_i)\) is nonempty and critical at \(p_i\). By Gromov's Compactness Theorem, \(M_i\rightarrow X\) where \(X\) is an \(n\)-dimensional Alexandrov space with curvature bounded below by \(k\). Let \(p := \lim p_i\). It follows that a subsequence  of the sequence of domains \(\{\seg(p_i)\}\) converges to a compact   subset \(\seg(p)\subset S^n_k\), and a subsequence of  the sequence of maps \(\{\exp_{p_i}:\seg(p_i)\rightarrow M_i\}\) converges to a surjective, 1-Lipschitz map \(\exp_p:\seg(\tilde p)\rightarrow X \), (see \cite{GrovePet2} or \cite{GrovePet3} for details). The set \(\seg(p)\subset S^n_k\) is star convex at a point \(\tilde p\in\seg(p)\) and the map \(\exp_p\) maps segments emanating from \(\tilde p\) in \(\seg(p)\) to segments emanating from \(p\) in \(X\). Conversely, any segment in \(X\) emanating from \(p\) is in the image under \(\exp_p\) of a segment in \(\seg(p)\) that emanates from \(\tilde p\). 

\begin{proposition}\label{seglens} \(\seg(p) = L^n_k(h,r)\)\end{proposition}
\begin{proof}
Set \(A_{r,h}(p) := \lim_i A_{r,h}(p_i)\). From the definition of eccentricity, it follows that for each \(q_i\in A_{r,h}(p_i)\) converging to \(q\in A_{r,h}(p)\), \(\cri_{p_i}(q_i)\rightarrow\cri_p(q)\).  Therefore, because every \(q_i\in A_{r,h}(p_i)\) satisfies \(\cri_{p_i}(q_i)\leq r\) and \(\dist(p_i,q_i)\leq h\), the same is true for every \(q\in A_{r,h}(p)\). In addition,  for every \(i\) the distance from the set \(A_{r,h}(p_i)\) is critical at \(p_i\), it follows that the distance from \(A_{r,h}(p)\) is critical at \(p\), i.e., \(\Uparrow_p^{A_{r,h}(p)}\subset \Sigma_p\) is a \(\pi/2\)-net. Therefore, since  for every \(i\), \[\seg(p_i)\subset \bigcap_{q_i\in A_{r,h}(p_i)}D^n_k( \tilde a_{v_{q_i}},\cri_{p_i}(q_i)),\]it follows that \[\seg(p)\subset \bigcap_{q\in A_{r,h}(p)}D^n_k( \tilde a_{v_q},\cri_p(q)).\]By Proposition \ref{volcorprop}, we can select points \(\{\tilde c_{v_{q}}\}\subset S(\tilde p,r-h)\) indexed over \( {v_{q}\in \Uparrow_p^{A_{r,h}(p)}}\) such that \[\bigcap_{q\in A_{r,h}(p)}D^n_k( \tilde a_{v_q},\cri_p(q))\subset \bigcap_{q\in A_{r,h}(p)}D^n_k( \tilde c_{v_q},r).\] Because \(\seg(p_i)\rightarrow \seg(p)\), it follows that,  \begin{eqnarray*}\vol L^n_k(h,r)&=&\lim\vol M_i
\\
&\leq& \lim \vol(\seg(p_i))\\
&=&\vol(\seg(p))\\
&\leq& \vol\left(\bigcap_{q\in A_{r,h}(p)}D^n_k( \tilde c_{v_q},r)\right)\\
&\leq& \vol L^n_k(h,r).
\end{eqnarray*} 
From the equality statement in Lemma \ref{gpvolsnk}, we must have \(\{\tilde c_{v_{q}}\}=\{\tilde c_{v_{q_1},}\tilde c_{v_{q_2}}\}\) where \(|\tilde c_{v_{q_1}}\tilde c_{v_{q_2}}| = 2(r-h).\) Therefore,\[\seg(p) \subset D^n_k(\tilde c_{v_{q_1}},r)\cap D^n_k(\tilde c_{v_{q_2}},r) = L^n_k(h,r).  \] From this and that \(\vol(\seg(p)) =\vol L^n_k(h,r)\), it follows that \(\seg(p) = L^n_k(h,r).\) \end{proof}

To continue,  we'll first fix notation for certain related geometric attributes of \(L^n_k(h,r)\). Let \(\tilde a_1,\tilde a_2,\tilde q_1,\tilde q_2\in S^n_k\) such that \(|\tilde a_1\tilde a_2| = 2(r-h)\), \(|\tilde p \tilde q_i| = h\), and \(|\tilde a_i\tilde q_i | = r\). We  assume that \[\seg(p) = L^n_k(h,r):=D^n_k(\tilde a_1,r)\cap D^n_k(\tilde a_2,r).\]
Denote the interior of \(L^n_k(h,r)\) by \(\mathring L^n_k(h,r)\). 

For \(i=1,2\), let \[D^{n-1}_i:=\partial D^n_k(\tilde a_1,r)\cap\partial L^n_k(h,r)\] and note \(D^{n-1}_i\) is a disk in the metric sphere  \(\partial D^n_k(\tilde a_i,r)\) centered at \(\tilde q_i\). Moreover, these disks have equal radii and  \(\partial L^n_k(h,r) = D^{n-1}_1\cup D^{n-1}_2\) (See Figure 3).

For distinct \(i,j\in \{1,2\}\), let \[B^{n-1}_i = \mathring D^{n-1}_i:=\{\tilde x\in S^n_k\mid|\tilde x\tilde a_i| < r\text{ and }|\tilde x\tilde a_j|= r\}\] be the interior of \(D^{n-1}_i\)

\begin{figure}
\begin{tikzpicture}[line cap=round,line join=round,>=triangle 45,x=1.0cm,y=1.0cm]
\clip(-3,-2.5) rectangle (3.3,3);
\draw [shift={(2,0)}]  plot[domain=2.3:3.98,variable=\t]({1*3*cos(\t r)+0*3*sin(\t r)},{0*3*cos(\t r)+1*3*sin(\t r)});
\draw [shift={(-2,0)}]  plot[domain=-0.84:0.84,variable=\t]({1*3*cos(\t r)+0*3*sin(\t r)},{0*3*cos(\t r)+1*3*sin(\t r)});
\draw [shift={(0,0)}]  plot[domain=0:3.14,variable=\t]({0*2.23*cos(\t r)+-1*0.2*sin(\t r)},{1*2.23*cos(\t r)+0*0.2*sin(\t r)});
\draw [shift={(0,0)},dotted]  plot[domain=-3.14:0,variable=\t]({0*2.23*cos(\t r)+-1*0.2*sin(\t r)},{1*2.23*cos(\t r)+0*0.2*sin(\t r)});
\draw [dotted] (-1,0.02)-- (1,-0.02);
\draw (1,-0.02)-- (2.79,-0.06);
\draw (-1,0.02)-- (-2.33,0.06);
\draw [shift={(0.45,0)}] plot[domain=-3.14:0,variable=\t]({0*1.7*cos(\t r)+1*0.16*sin(\t r)},{-1*1.7*cos(\t r)+0*0.16*sin(\t r)});
\draw [shift={(0.45,0)},dotted]  plot[domain=0:3.14,variable=\t]({0*1.7*cos(\t r)+1*0.16*sin(\t r)},{-1*1.7*cos(\t r)+0*0.16*sin(\t r)});
\draw (0,.22) node{\tiny\(\tilde p\)};
\draw (3,0) node{\(\tilde s\)};
\draw (1.15,0.15) node{\tiny\(\tilde q_1\)};
\draw (1,-0.045) node{\huge\(\cdot\)};
\draw (0,-.028) node{\huge\(\cdot\)};
\draw (0.55,0.19) node{\tiny\(\tilde s(t)\)};
\draw (0.47,-0.04) node{\huge\(\cdot\)};
\draw (-.99,.0) node{\huge\(\cdot\)};
\draw (-.82,.18) node{\tiny\(\tilde q_2\)};
\draw (-1.97,.02) node{\huge\(\cdot\)};
\draw (-1.97,.18) node{\tiny\(\tilde a_1\)};
\draw (1.97,0.1) node{\tiny\(\tilde a_2\)};
\draw (1.97,-.06) node{\huge\(\cdot\)};
\draw(1.9,-1.5) node{\small\(D^{n-1}_1\)};
\draw(-1.6,-1.5) node{\small\(D^{n-1}_2\)};
\draw(1.8,1.5) node{\small\(S_t\)};
\draw[-latex'] (-1.3,-1.5)to[out= -45, in = 165](-.4,-1.5);
\draw[latex'-] (0.1,-1.5) to[out= -45, in = -185](1.3,-1.5);
\draw[latex'-] (.4,1.5) to[out= -45, in = 18](1.3,1.5);
\end{tikzpicture}
\caption{}
\end{figure}

Let \(\tilde s:\R\rightarrow S^n_k\) be the geodesic through \(\tilde a_1\) and \(\tilde a_2\) such that \(\tilde s(-h) = \tilde q_2\), \(\tilde s(0) = \tilde p\), and \(\tilde s(h) = \tilde q_1\). For each \(t\in [-h,h]\), let \(H_t\) be the totally geodesic hyperplane in \(S^n_k\) through \(\tilde s(t)\) and orthogonal to \(\tilde s'(t)\) and let \[S_t := H_t\cap\partial L^n_k(h,r)\] be the \((n-2)\)-dimensional metric sphere in \(H_t\).  Note that \(S_0 = \partial D^{n-1}_1=\partial D^{n-1}_2 =D^{n-1}_1\cap D^{n-1}_2\). 

We recall an observation made in \cite{GrovePet3}. Let \(M\) be a compact, Riemannian \(n\)-manifold with \(\sec M\geq k\in \R\). Let \(p\in M\) be a point that realizes the radius of \(M\) and let \(Q\subset M\) and \(r:Q\rightarrow \R^+\)  a function. If \(\tilde p\) is the origin of \(\seg(p)\), and \(\tilde Q := \exp_p^{-1}(Q)\subset S^n_k\), the so called ``Swiss Cheese" volume comparison given in  \cite{GrovePet3} says, \[\vol \left(M-\bigcup_{q\in Q}B(q,r(p))\right)\leq \vol \left(D^n_k(\tilde p,\rad M)-\bigcup_{\tilde q\in \tilde Q} B(\tilde q,r\circ\exp_p(\tilde q))\right).\]

Now let \(\tilde p_i\) be the origin of \(\seg(p_i).\) By Proposition \ref{volcorprop}, for each \(i\), \[\seg(p_i)\subset I(\tilde p_i,r):= \bigcap_{q_i\in A_{r,h}(p_i)}D^n_k( \tilde c_{v_{q_i}},r),\] and a straightforward modification of the above shows\begin{eqnarray}\label{cheesier}\qquad\vol \left(M_i-\bigcup_{q\in Q}B(q,r(p))\right)&\leq& \vol \left(I(\tilde p_i,r)-\bigcup_{\tilde q\in \tilde Q} B(\tilde q,r\circ\exp_p(\tilde q))\right).\end{eqnarray}

 \begin{lemma} \label{toplemma} The map \(\exp_p:L^n_k(h,r)\rightarrow X\) satisfies \begin{enumerate}\item \(\left.\exp\right|_{\mathring L^n_k(h,r)}\) is injective,\item \(\exp_p\) preserves the length of paths,\item \(\left.\exp_p\right|_{B^{n-1}_i}\) is at most \(2\) to \(1\), and
\item there is a positive integer \(c(n,k,h,r)\) such that \(\left.\exp_p\right|_{S_0}\) is no more than \(c\) to \(1\). 
\end{enumerate}
\end{lemma}           
\begin{proof} 

   Up to needing Equation (\ref{cheesier}), the proofs of Parts (1) - (3)  are identical to the proofs of the analogous parts of Lemma 2.5 in \cite{GrovePet3}.  We give the proof of Part (4) which is similar to the proof of part (3) (cf. \cite{NL}). 

Let \(\tilde q\in S^{}_0\) and \(\rho>0\). Let \(B(\tilde q,\rho)\) be a metric ball in \(S^n_k\) centered at \(\tilde q\) of radius \(\rho\). By Bishop-Gromov, the function\[\rho\rightarrow\frac{\vol B(\tilde q, \rho)}{\vol (B(\tilde q, \rho)\cap L^n_k(h,r))}\] is nondecreasing. This function is also bounded below by 1, so let \(c(n,k,h,r)\) be the smallest integer larger than \[\lim_{\rho\rightarrow0}\frac{\vol B(\tilde q, \rho)}{\vol (B(\tilde q, \rho)\cap L^n_k(h,r))}.\]
By symmetry, \(c\) is independent of \(\tilde q\). 

Let \( q\in X\). For a contradiction, suppose there are  \(c+1\) distinct points \(\{\tilde x^k\}_{k=1}^{c+1}\subset\partial L^n_k(h,r)\) such that \(\exp_p(\tilde x^k) =  q  \). Choose \(\rho>0\)  so that \(\{B(\tilde x^k,\rho)\}_{k=1}^{c+1}\subset S^n_k\) is a disjoint collection. For each \(i\), let \(\{\tilde x^k_i\}_{k=1}^{c+1}\subset \seg(p_i)\) be chosen  so that \(\lim_i\{\tilde x^k_i\} = \{\tilde x^k\}\) and \(\lim_i \exp_{p_i}(\tilde x^k_i) = \exp_{p}(\tilde x^k) = q\).

Note for any \(\lambda\in (0,1)\) and any   \(k\in\{1,\ldots,c+1\}\), if \(i \) sufficiently large, by Proposition \ref{seglens}, \[\lambda\vol (B(\tilde x^k,\rho)\cap L^n_k(h,r))\leq\vol (B(\tilde x^k_i,\rho)\cap I(\tilde p_i,r)) .\] Therefore, if  \(i\) is large enough, by Equation (\ref{cheesier}), the display above, and the definition of \(c\), \begin{eqnarray*}\vol M_i -\vol\left(\bigcup _{k=1}^{c+1}B(\exp_{p_i}(\tilde x^k_i),\rho)\right) &\leq &\vol I(\tilde p_i,r) - \sum_{k=1}^{c+1}\vol (B(\tilde x^k_i,\rho)\cap I(\tilde p_i,r))\\ &\leq& \vol I(\tilde p_i,r) - \sum_{k=1}^{c+1}\lambda\vol (B(\tilde x^k,\rho)\cap L^n_k(h,r))\\&\leq& \vol I(\tilde p_i,r) - \frac{c+1}{c}\lambda\vol(B(\tilde q,\rho)).  \end{eqnarray*}
However,  \(\bigcup _{k=1}^{c+1}B(\exp_{p_i}(\tilde x^k_i),\rho)\rightarrow B(q,\rho)\), and both \(\vol M_i\) and \(\vol I(\tilde p_i,r)\) converge to \( \vol L^n_k(h,r)\). So, if \(\lambda\) is sufficiently close to \(1\), and \(i\) is sufficiently large, the above inequality provides the desired contradiction.       
\end{proof}
Let \(R_{H_0}:S^n_k\rightarrow S^n_k\) be  reflection over the hyperplane \(H_0\). Note that \(R_{H_0}(D^{n-1}_i)= D^{n-1}_j\) for \(i\neq j\in\{1,2\}.\) Equip \(\partial L^n_k(h,r)\) with the induced length metric from \(S^n_k\). Note this metric restricted to either \(D^{n-1}_1\) or \(D^{n-1}_2\) is Riemannian of constant curvature and with these metrics \(D^{n-1}_1\)and \(D^{n-1}_2\) are isometric.    

\begin{lemma}\label{identification} There is a positive integer \(m\) and an isometry \(\phi: \partial L^n_k(h,r)\rightarrow \partial L^n_k(h,r)\) which fix \(\tilde q_1, \tilde q_2\), leaves \(S_0\) invariant, and satisfies \(\phi^m=\mathrm{id}\). Moreover, \(\exp_p:L^n_k(h,r)\rightarrow X\) induces an isometry between \(X\) and either \begin{enumerate}\item \(L^n_k(h,r)/(\tilde u\sim \phi(\tilde u))\) provided \(\phi\) is an involution, or 
\item \(L^n_k(h,r)/(\tilde u\sim R_{H_0}\circ\phi(\tilde u))\)\end{enumerate}
where \(\tilde u\in \partial L^n_k(h,r)\)\end{lemma}  

\begin{proof}
As in  \cite{GrovePet3}, by Part 3 of Lemma \ref{toplemma}, we can define a map \(f:B_1^{n-1}\cup B_2^{n-1}\rightarrow \partial L^n_k(h,r)\) by \begin{eqnarray*}f(\tilde u) =\left\{\begin{array}{cc}
\tilde u & \text{~if~} \exp^{-1}_p(\exp_p(\tilde u)) = \{\tilde u\} \\

\tilde v & \text{~if~} \exp_p(\tilde u) = \exp_p(\tilde v), \tilde u\neq\tilde v \\
\end{array}\right. . \end{eqnarray*}

As noted in \cite{GrovePet3}, the map \(f\) is continuous as a point of discontinuity would produce a bifurcation of geodesics in
\(X\). By Part 2 of Lemma \ref{toplemma}, \(f\) is \(1\)-Lipschitz, so it uniquely extends to a continuous map \(f: \partial L^n_k(h,r)\rightarrow \partial L^n_k(h,r)\).

Assume first that \(r-h>0\). In this case, \(\lambda_{\tilde p}(\tilde q_1)=\lambda_{\tilde p}(\tilde q_2)>0\). In particular, if \(\tilde x\in D^{n-1}_i\), by Proposition \ref{proposition1}, \(\m(|\tilde p\tilde x|) = \lambda_{\tilde p}(\tilde q_i)\m(|\tilde q_i\tilde x|) - \m(|\tilde p\tilde q_i|)\). Therefore, if \(\tilde x,\tilde y\in \partial L^n_k(h,r)\), then \(|\tilde p\tilde x| = |\tilde p \tilde y|\) if and only if for some \(t\in[-h,h]\), either \(\tilde x,\tilde y \in S_t\) or \(\tilde x\in S_t\) and \(\tilde y\in S_{-t}\). By Parts 1 and 2 of Lemma \ref{toplemma}, it follows that for all \(t\in[0,h]\), either \(f(S_t) = S_t\) or \(f(S_t) = S_{-t}\). Therefore, by continuity of \(f\), it follows that \(f(D^{n-1}_i) = D^{n-1}_j\) for either \(i=j\in\{1,2\}\) or \(i\neq j\in\{1,2\}\).

In the case \(f(D^{n-1}_i) = D^{n-1}_i\) for \(i\in \{1,2\}\) set \(\phi =f\). By Part 3 of Lemma \ref{toplemma}, it follows that \(\phi^2 = \mathrm{id}\), in particular, \(\phi\) must be an isometry.
By Part 1 of Lemma 5, it follows that \(X\) is isometric to \(L^n_k(h,r)/(\tilde u\sim \phi(\tilde u)\)). 

In the case \(f(D^{n-1}_i) = D^{n-1}_j\) for \(i\neq j\in \{1,2\}\), set \(\phi = R_{H_0}\circ f\).  For each \(\tilde u \in S_0\), \(\exp_p^{-1}(\exp_p(\tilde u)) = \{\phi^{k}(\tilde u)\mid k\in \mathbb{N}\}\),  so it follows from Part 4 of Lemma \ref{toplemma} this set can consist of no more that \(c(n,k,h,r)\)-elements. It follows that  \(\phi^m=\mathrm{id}\) for some positive integer \(m\) possibly larger than \(c(n,k,h,r)\).  In particular, \(\phi\) must be an isometry. Again by Part 1 of Lemma 5,  \(X\) is isometric to \(L^n_k(h,r)/(\tilde u\sim R_{H_0}\circ\phi(\tilde u))\).

When \(r=h\) we have \(L^n_k(h,r) = D^n_k(r)\). This case is handled by Lemma 2.6 in   \cite{GrovePet3} where they show that the identification must occur via an isometric involution.  Up to an isometry of \(D^n_k(r) =L^n_k(r,r)\), the conclusion is the same.    \end{proof}
\begin{lemma}\label{manifoldlemma} Let \(c(n,k,h,r)\) be as in Lemma \ref{toplemma}. If \(h<r, \) there is a totally geodesic hyperplane \(P\subset S^n_k\) that contains the geodesic through \(\tilde a_1\) and \(\tilde a_2 \), such that if \(R_P:S^n_k\rightarrow S^n_k\) is reflection over \(P\), then \(X\) is isometric to either 
\begin{enumerate}[(A)]\item \(P^n_k(h,r) = L^n_k(h,r)/(\tilde u\sim R_P(\tilde u))\),   \item\(L^n_k(h,r,\mathrm{id}) = L^n_k(h,r)/(\tilde u\sim R_{H_0}(\tilde u))\),  or \item \(L^n_k(h,R,\phi_m)=L^n_k(h,r)/(\tilde u\sim R_{H_0}\circ\phi_m(\tilde u))\),  and \begin{enumerate}\item \(\phi_m:\partial L^n_k(h,r)\rightarrow \partial L^n_k(h,r) \) is an isometry which leaves \(S_0\) invariant,
\item \(\phi_m\)  has order \(m\in\{2,\ldots,c(n,k,h,r)\}\), and 
\item  the cyclic group \(\mathbb{Z}_m =\langle\phi_m\rangle \) acts freely and orthogonally on \(S_0\) \end{enumerate}

\end{enumerate}
where \(\tilde u\in \partial L^n_k(h,r)\)
\end{lemma}
\begin{proof}
 Given that \(X\) is the Gromov-Hausdorff limit of a sequence of Riemannian \(n\)-manifolds with an upper diameter bound, lower curvature bound, and lower volume bound, by Perelman's Stability Theorem, it follows that \(X\) is a topological \(n\)-manifold (see \cite{VK}, Lemma 3.2).

By Lemma \ref{identification}, there are two possibilities for the isometry type of \(X\). 

\underline{Case 1:} \(X\) is isometric to 
\(L^n_k(h,r)/(\tilde u\sim \phi(\tilde u))\) where \(\phi:\partial L^n_k(h,r)\rightarrow \partial L^n_k(h,r)\) is an isometric involution that fix \(\tilde q_1\) and \(\tilde q_2\). 

Identify \(L^n_k(h,r)\) with the unit disk \(D^n\subset\R^n=\R\oplus J\) and the isometry \(\phi\) with a linear involution \(\phi_2:\R^n\rightarrow \R^n\) such that \(\phi_2(J) = J\).  

As in Lemma 2.7 of \cite{GrovePet3}, because \(X\) is a manifold, there are only two possibilities: \(\phi_2=-\mathrm{id}\) or \(\phi_2=R_J\) where \(R_J:\R^n\rightarrow \R^n\) is reflection over \(J\). This follows, as observed in \cite{GrovePet3}, that since \(\phi_2\) is an isometric involution we can assume for some \(j\), \(\phi_2|_{\R^j\times\{0\}} = \mathrm{id}\) and \(\phi_2|_{\{0\}\times{\R^{n-j}} }= -\mathrm{id}\). As \(X = D^n/(\tilde u\sim \phi_2(\tilde u))\) where \(\tilde u\in \partial D^n\), \(X\) must be homeomorphic to the \(j\)-fold suspension \(\Sigma^j\R P^{n-j}\) which  has the homology of a manifold only when \(j = 0\) or \(n-1\).  
 
 Because the isometry \(\phi\) in Lemma \ref{identification} fix \(\tilde q_1\) and \(\tilde q_2\), it follows that \(\phi = R_P\) where \(P\) is a hyperplane that, if \(h<r\), contains the geodesic through \(\tilde a_1\) and \(\tilde a_2\). This gives Part (A).

\underline{Case 2:} \(X\) is isometric to 
\(L^n_k(h,r)/(\tilde u\sim R_{H_0}\circ \phi(\tilde u))\) where \(\phi:\partial L^n_k(h,r)\rightarrow \partial L^n_k(h,r)\) is an isometry that has finite order and fix \(\tilde q_1\) and \(\tilde q_2\).

Let \(C\) be the unit circle in \(\R^2=\mathbb{C}\). For \(j\in \{1,2,\ldots, m\}\), let \(s_j := \{e^{i\theta}\in C\mid 2\pi\frac{j-1}{m}\leq \theta\leq 2\pi\frac{j}{m}\} \). Let \(J =\R^{n-1} \) and identify  \(S_0\subset H_0\) with the unit \((n-2)\)-sphere \(S_J^{n-2}\) in \(J\) and the isometry \(\phi\) with an isometry \(\phi_m\in O(n-1)\) that satisfies \((\phi_m)^m=\mathrm{id}\) for some positive integer \(m\).  

Now let \(\varphi_m:\mathbb{C}\oplus J\rightarrow \mathbb{C}\oplus J\) be an isometry of \(\R^{n+1}\)  defined by \[\varphi_m(z,x) = (e^{\frac{2\pi}{m}i}z, \phi_m(x)).\]
By topologically identifying \(L^n_k(h,r)\) with \(s_1\ast S_J^{n-2}\), it follows  from the definition of \(\varphi_m\) and Part (2) of Lemma \ref{identification} that 
\begin{enumerate}
\item \(\varphi_m\) has order \(m\),
\item the cyclic group \(\mathbb{Z}_m = \langle\varphi_m\rangle\) acts orthogonally on the round sphere \(S^n := C\ast S_J^{n-2}\) where \(\ast\) is the spherical join. 
\item \(X =L^n_k(h,r)/(\tilde u\sim R_{H_0}\circ \phi(\tilde u))\)  is homeomorphic to the quotient \(S^n/\langle\varphi_m\rangle.\)\end{enumerate}

Let \(\pi:S^n\rightarrow S^n/\langle\varphi_m\rangle\) be the projection map. If, in addition, the action of \(\langle\varphi_m\rangle\) on \(S^n\) is free, by Part (4) of Lemma \ref{toplemma} we have \(m\in\{2\ldots,c(n,k,h,r)\}\) and so Part (C) holds.
In particular, \(X\) is homeomorphic to a Lens space \(S^n/\mathbb{Z}_m\). 

Suppose there is a point \(p\in S^n\) for which the isotropy group \(G_1 = \langle\varphi_m\rangle_p\neq\mathrm{id}.\) Since the action of \(G_0 =\langle\varphi_m\rangle\) on \(S^n\) is linear,  \(G_1\) acts orthogonally on the (\(n-1\))-sphere \(S^{n-1}\subset S^n\) at distance \(\pi/2\) away from \(p\).  We can then identify the space of directions  at \(\pi(p)\in S^n/G_0\) with the quotient \(S^{n-1}/G_1\). Since \(S^n/G_0\) is a manifold,  it follows that this quotient \(S^{n-1}/G_1\) is homeomorphic to \(S^{n-1}\) -- which is also a manifold. Iterating this procedure we obtain a sequence of subgroups \(G_k<\cdots<G_1<G_0\)  where  \(G_k\neq \mathrm{id}\) has the following properties: \begin{enumerate}\item \(G_k\) fixes, point-wise, a totally geodesic \(S^{k-1}\subset S^n\), \item \(G_k\) acts freely and orthogonally on a totally geodesic \(S^{n-k}\subset S^n\) at distance \(\pi/2\) away from the fixed set \(S^{k-1}\), \item the quotient  \(S^{n-k}/G_k\) is homeomorphic to \(S^{n-k}\). \end{enumerate}
Since \(G_k\) is cyclic,   we must have \(S^{n-k} = S^1\).  In particular, \(k=n-1\),   and since \(\varphi_m(z,x) = (e^{\frac{2\pi}{m}i}z, \phi_m(x))\), we  have \(\phi_m = \mathrm{id}\). So Part (B) holds. \end{proof}

\begin{proof}[\bf{Proof of Part 2 of Theorem \ref{mainthrm}} (topological version)] The identification spaces \(P^n_k(h,r)\) and  \(L^n_k(h,r,\mathrm{id})\), topologically, are spheres. The proof of Lemma \ref{manifoldlemma} shows \(L^n_k(h,r,\phi_m)\) is homeomorphic to a quotient of \(S^n\) by a free and orthogonal action of \(\mathbb{Z}_m\) where \(m\leq c(n,h,h,r),\) thus is a Lens space. The proof now follows from  Gromov's Compactness Theorem, Perelman's Stability Theorem, and Lemma \ref{manifoldlemma}.    \end{proof} 
\section{Smooth Perturbation of the Limits}

To construct Riemannian metrics that satisfy the hypotheses of Theorem \ref{mainthrm}, we give smooth perturbations of \(L^n_k(h,r,\phi_{m})\) and \(P^n_k(h,r)\), for any \(\phi_{m}\in C(n-2,m)\) as in Example \ref{example1}.
   
   \begin{proof}[\bf{Proof of Part 3 of Theorem \ref{mainthrm}}]   

For perturbations of \(L^n_k(h,r,\mathrm{id})\) and \(P^n_k(h,r)\) we follow \cite{GrovePet3}. For the hyperplane \(J = H_0\) or \(P\), define \(L^{n,+}_{k,J}(r)\) to be one side of \(L^n_k(h,r)\) separated by \(J\) (see Figure \ref{figureP} for \(L^{n,+}_{k,P}(r))\). Isometrically embed  \(L^{n,+}_{k,J}(r)\) into a totally geodesic \(S^n_k\subset S^{n+1}_k\) and take boundaries of smooth, symmetric, convex neighborhoods of  \(L^{n,+}_{k,J}(r)\).

  Let  \(\phi_m\in C(n-2,m)\). For perturbations of \(L^n_k(h,r,\phi_{m})\),  since \(m\leq c(n,k,h,r)\) we can find \(m\) points \(\tilde p_1,\ldots,\tilde p_m\in S^2_k\) that lie in a circle \(C:=S(\tilde q, \rho)\subset S^2_k\) such that \(\rho\leq \frac{1}{2}\diam S^n_k\) is a real number and \(|\tilde p_i\tilde p_{i+1}|=2h\) (indices mod \(m)\). Let \(C^\ast\) be the circle of length \(2mh\) in \(S^2_k\) formed by joining the segments between \(\tilde p_i\) and \(\tilde p_{i+1}\). 
 
Let \(R\in(0,\frac{1}{2}\diam S^n_k]\) be the intrinsic radii of the boundary disks \(D^{n-1}_i\) of \(L^n_k(h,r)\). Let \(\diam S_0\) be the intrinsic diameter  of \(S_0=D^{n-1}_1\cap D^{n-1}_2\) with metric induced from \(D^n_k(\tilde a_1,r)\). For  \(\varepsilon>0\), by using a doubly warped product metric on \([0,R-\varepsilon]\times C\times S^{{n-2}}\) we can construct a smooth metric \(g\) on \(S^n=C\ast S^{n-2}\) such that induced metric on the \(S^{n-2}\) factor  has constant curvature and diameter \(\diam S_0 -\tau(\varepsilon)\) where \(\tau(\varepsilon)\searrow 0\) as \(\varepsilon\rightarrow0\).

Smoothly deform \(g\) on \(S^n\) to a metric \(g_\varepsilon\)  such that outside of an \(\varepsilon\)--neighborhood of \(\{p_1,\ldots,p_m\}\ast S^{n-2}\subset C\ast S^{n-2}\), \(g_\varepsilon\) has constant curvature \(k\).   This gives a smooth metric \(S^n_\varepsilon:=(S^n,g_\varepsilon)\) on \(S_\varepsilon^1\ast S^{n-2}\)  where \(S^1_\varepsilon\) is a circle in \(S^2_k\) that contains the points \(\tilde p_1,\ldots,\tilde p_m\) for every \(\varepsilon>0\) and converges to \(C^\ast\) as \(\varepsilon \rightarrow0\). So, if \(\mathbb{Z}_m\) acts on the \(S^1_\varepsilon\) factor by taking \(\tilde p_i\) to \(\tilde p_{i+1}\), and by \(\langle \phi_{m_i}\rangle\) on the \(S^{n-2}\) factor we obtain a smooth Riemannian metric \(g_\varepsilon\) on a Lens space \(S^n/\mathbb{Z}_m\) with fundamental domain  converging  to \(L^n_k(h,r,\phi_{m})\) as \(\varepsilon\rightarrow 0\). It follows that \(\vol S^n_\varepsilon/\mathbb{Z}_m\rightarrow \vol L^n_k(h,r)\). Note that \( X  = \lim_{\varepsilon\rightarrow 0} S^n_{\varepsilon}\) will be Riemannian if and only if \(C=C^\ast\),  \(C^\ast\)has diameter \(\diam S^n_k\), and \(S_0 =S^{n-2}_1\). In particular, if a Riemannian manifold \(M\) satisifies \(\sec M\geq k, \rad M\leq r\),  \(\sag_rM\leq h\),  and \(\vol M=\vol L^n_k(h,r) \), then \(k>0\), \(r=\frac{1}{2}\diam S^n_k\) and for some \(m\leq c(n,k,h,r)\), \(hm = \frac{1}{2}\diam S^n_k\).   
  \end{proof}

%+Bibliography

%-Bibliography

\end{document}